\documentclass[11pt,reqno]{amsart}
\usepackage[english]{babel}
\usepackage[utf8]{inputenc}
\usepackage[T1]{fontenc}
\usepackage{lmodern} \normalfont 
\DeclareFontShape{T1}{lmr}{bx}{sc} { <-> ssub * cmr/bx/sc }{}
\usepackage{microtype}	

\usepackage{comment}

\usepackage{amssymb}
\usepackage{amsmath}
\usepackage{amsthm}
\usepackage{mathtools}
\usepackage{mathrsfs}
\usepackage{accents} 
\usepackage{etoolbox}
\usepackage{siunitx}
\usepackage{dsfont} 
\usepackage{graphicx}
\usepackage{color}
\usepackage[dvipsnames]{xcolor}
\usepackage{tikz}
\usetikzlibrary{plotmarks}
\usetikzlibrary{calc,positioning,shapes}
\usetikzlibrary{patterns,decorations.pathmorphing,decorations.markings}
\usepackage{pgfplots}
\pgfplotsset{compat=newest}
\usepackage[margin=10pt,font=small,labelfont=bf,labelsep=endash]{caption}
\usepackage{subcaption}
\usepackage{nicematrix} 

\usepackage{booktabs}
\usepackage{paralist}
\usepackage{soul}

\numberwithin{equation}{section}

\usepackage{algorithm}
\usepackage{algorithmicx}
\usepackage{algpseudocode}

\usepackage{cite}
\usepackage{multirow} 
\usepackage{enumitem} 
\setlist[enumerate]{label=(\roman*)}

\theoremstyle{plain}
\newtheorem{theorem}{Theorem}[section]
\newtheorem{proposition}[theorem]{Proposition}
\newtheorem{lemma}[theorem]{Lemma}
\newtheorem{corollary}[theorem]{Corollary}
\newtheorem{remark}[theorem]{Remark}
\newtheorem{definition}[theorem]{Definition}

\usepackage[colorlinks,citecolor=black,urlcolor=black]{hyperref}
\hypersetup{allcolors=black}

\usepackage[nameinlink,noabbrev]{cleveref}

\crefname{subsection}{subsection}{subsections}
\Crefname{subsection}{subsection}{subsections}

\crefname{theorem}{theorem}{theorems}
\Crefname{theorem}{Theorem}{Theorems}

\crefname{proposition}{proposition}{propositions}
\Crefname{proposition}{Proposition}{Propositions}

\crefname{lemma}{lemma}{lemmas}
\Crefname{lemma}{Lemma}{Lemmas}

\crefname{corollary}{corollary}{corollaries}
\Crefname{corollary}{Corollary}{Corollaries}

\crefname{remark}{remark}{remarks}
\Crefname{remark}{Remark}{Remarks}

\crefname{definition}{definition}{definitions}
\Crefname{definition}{Definition}{Definitions}

\crefname{assumption}{assumption}{assumptions}
\Crefname{assumption}{Assumption}{Assumptions}

\crefname{example}{example}{examples}
\Crefname{example}{Example}{Examples}

\newcommand{\R}{\mathbb{R}}

\DeclareMathOperator*{\argmin}{arg\,min} 
\DeclareMathOperator*{\esssup}{ess\,sup} 

\definecolor{mycolor1}{rgb}{0.00000,0.44700,0.74100}
\definecolor{mycolor2}{rgb}{0.85000,0.32500,0.09800}
\definecolor{mycolor3}{rgb}{0.92900,0.69400,0.12500}
\definecolor{mycolor4}{rgb}{0.46600,0.67400,0.18800}
\definecolor{mycolor5}{rgb}{0.49400,0.18400,0.55600}

\newcommand{\matlab}{MATLAB\textsuperscript{\textregistered}}

\title{Risk averse deterministic Kalman filters for uncertain dynamical systems}

\author{Karl Kunisch${}^{\star,\dagger}$ \and Jesper Schr\"oder${}^\dagger$}
\address{${}^{\star}$ Institute of Mathematics and Scientific Computing, University of Graz, A-8010 Graz, Austria}
\address{${}^{\dagger}$ Johann Radon Institute, Austrian Academy of Sciences, A-4040 Linz, Austria}
\email{karl.kunisch@uni-graz.at}
\email{jesper.schroeder@ricam.oeaw.ac.at}
\date{\today}

\begin{document}
%
\begin{abstract}
    Taking a deterministic viewpoint this work investigates  extensions of the Kalman-Bucy filter for state reconstruction  to systems containing
    parametric uncertainty in the state operator.
    The emphasis lies on risk averse designs reducing the probability of large reconstruction errors. In a theoretical analysis error bounds in terms of the variance of the uncertainties are derived. The article concludes with a numerical implementation of two examples allowing for a comparison of risk neutral and risk averse estimators. 
\end{abstract}
%

\maketitle


%
\section{Introduction}
%
This work proposes estimators for state reconstruction of linear perturbed dynamical systems of the form 
\begin{equation*}
	\begin{aligned}
		\dot{x}(t) &= A_\sigma x(t) + B v(t)
		~~~~~
		t \in (0,T),\\
		x(0) &= x_0 + \eta,
	\end{aligned}
\end{equation*}
with unknown, deterministic disturbances $v \in L^2(0,T;\mathbb{R}^m)$ and $\eta \in \mathbb{R}^n$. The system matrix $A_\sigma$ depends on a parametric uncertainties indicated by $\sigma \in \Sigma$. The objective is an online approximation of the system's state $x$ based on incomplete measurements given as 
\begin{equation*}
	y(t) = C x(t) + \mu(t)
	~~~~~
	t \in (0,T),
\end{equation*}
that are affected by the output disturbance $\mu \in L^2(0,T;\mathbb{R}^r)$.

Utilizing flawed measurements to estimate the state of systems suffering from perturbances is of high interest in the applied sciences at least since Wiener's seminal work \cite{Wie49}.
Kalman's and Bucy's work in \cite{KalB61} is  tailored to the linear setting, for known parameter $\sigma$,  and modeling the disturbances as Gaussian white noise, see also \cite{Kal60} for the time-discrete analogue. The so-called Kalman-Bucy filter is easy to implement, numerically efficient, and widely applied.

Since the parametric uncertainty introduces nonlinearities and multiplicative noise, it is not directly applicable to the class of systems considered here. In this work a risk averse extension is proposed based on Mortensen's strategy of recasting the Kalman filter as the energy minimal state \cite{Mor68}. Specifically, we study probability measures with respect to the parameter of the energy. The measures of interest are the entropic risk measure and the essential supremum for which the minimizers are characterized based on first order optimality conditions. The proposed measures utilize the energies for each system corresponding to a parameter $\sigma \in \Sigma$ separately, followed by a variational criterion which determines a universal state-estimator for the ensemble.
An important part of our work consists in establishing estimates on the difference between the proposed state estimator and that estimator which would be optimal if we knew the exact parameter $\sigma \in \Sigma$.
For both measures we analyze their temporal regularity. The entropic risk measure depends on a risk aversion parameter $\theta > 0 $. We verify that for $\theta \to 0$ the resulting estimator converges to the risk neutral estimator and for $\theta \to \infty$ we obtain convergence to the minimizer of the essential supremum.
In the numerical part of our work the proposed state estimators are compared to  results emerging from the risk neutral approach associated with the expectation which we investigated in earlier work \cite{KunSc25}. 

Let us further comment on the two measures which we investigate.
The entropic risk measure was introduced in the context of financial mathematics to optimize portfolios with a flexible level of risk aversion \cite{FoeSchi16}. More recently it was adapted for optimal control under uncertainty \cite{GutEtAl24,GutKu25} leading to risk averse control strategies reducing the likelihood of large costs. This functional also arises in many other applied sciences including machine learning and neural networks \cite{CalGaPo20}, or statistical mechanics, \cite{PatBe11}, for example. 
It is well known that for infinitely large risk aversion the entropic risk measure corresponds to the essential supremum which in the framework of this work turns into a maximum. The associated estimator is hence given by the solution of a min-max problem which on a technical level is closely related 
to the minimum ball problem arising in operations research.  In the 2-dimensional case this problem  has a long history, see for instance \cite{HV82} and the references provided there. The $\mathbb{R}^n$ case was investigated for instance in \cite{CaDea24,DeaCa23}. These papers also contain many references on the algorithmic developments for this class of problems.  In this paper, to analyze the min-max formulation as well as the entropic risk formulation we follow a different analytical route  by treating these problems by convex analysis techniques. In this way we can use results from  \cite{MorNaVi13} to obtain  existence and uniqueness of solutions and from \cite{NamNgSa12} to derive optimality conditions, in an elegant manner.
The latter are of vital importance for the error analysis which we provide.
%
\section{The Kalman filter}
%
We begin with a brief presentation of two different derivations of the Kalman-Bucy filter for time-invariant systems. While our proposed estimators rely only on the deterministic formulation a recap of the stochastic formulation enriches the interpretation of our approach. 
We include these well known concepts for the purpose of a self-contained work.
%
\subsection{The stochastic formulation}\label[subsection]{subsec: stochForm}
%
We illustrate the stochastic filter based on the original work \cite{KalB61} and refer to \cite{Oks98,Xio08} for a more rigorous treatment in terms of SDEs.  
Consider the disturbed system with system state $X_t$ and measured output $Y_t$ modeled by
\begin{equation}\label{eq: stochSys}
\begin{aligned}
    \tfrac{\mathrm{d}}{\mathrm{d}t} X_t &= A X_t + B V_t
    ~~~~~ t \in (0,T),\\
    Y_t &= C X_t + W_t,
    ~~~~~ t \in (0,T),
\end{aligned}
\end{equation}
with system matrix $A \in \mathbb{R}^{n,n}$, noise input matrix $B \in \mathbb{R}^{n,m}$, and output matrix $C \in \mathbb{R}^{r,n}$.
The dynamics and the measurement are disturbed by the multivariate Gaussian random variables $V_t \in \mathbb{R}^m$ and $W_t \in \mathbb{R}^r$, respectively. The so-called white noise terms have zero mean and covariances given by 
\begin{equation*}
\begin{aligned}
    \text{cov}[V_t,V_s] &= R \, \delta(t-s),\\
    \text{cov}[W_t,W_s] & = Q \, \delta(t-s),
\end{aligned}
\end{equation*}
where $R \in \mathbb{R}^{m,m}$ and $Q\in \mathbb{R}^{r,r}$ are given  symmetric, positive definite matrices. 
The distribution of the state at time $t$ is Gaussian and denoted by $X_t$, where the initial distribution $X_0$ is characterized via its mean $x_0 \in \mathbb{R}^n$ and covariance $\Gamma \in \mathbb{R}^{n,n}$, symmetric, positive definite. The two noise terms and the initial distribution are assumed to be mutually independent and finally $Y_t$ denotes the distribution of the output at time $t$.

According to \cite{KalB61} the goal is to optimally estimate the realized state $x(t)$ of the system given realized, measured outputs $y(s)$, $0 \leq s \leq t$.
In their work Kalman and Bucy show that the optimal estimate $\widehat{x}(t)$ is given via
\begin{equation}\label{eq: stochFil}
\begin{aligned}
    \dot{\widehat{x}}(t) &= A \widehat{x}(t) + \Pi(t) C^\top Q^{-1} (y(t) - C\widehat{x}(t))
    ~~~~~t \in (0,T),\\
    \widehat{x}(0) &= x_0,
\end{aligned}
\end{equation}
where $\Pi$ is the unique solution to the differential Riccati equation 
\begin{equation}\label{eq: stochRicc}
\begin{aligned}
    \dot{\Pi}(t) &= A \Pi(t) + \Pi(t) A^\top - \Pi(t) C^\top Q^{-1} C \Pi(t) + B R B^\top
    ~~~~~ t \in (0,T),\\
    \Pi(0) &= \Gamma.
\end{aligned}
\end{equation}
In the stochastic interpretation, given a measured output the state of the system at time $t$ is estimated via the Gaussian distribution $\mathcal{N}(\widehat{x}(t),\Pi(t))$ where the mean and covariance are characterized via \eqref{eq: stochFil} and \eqref{eq: stochRicc}, respectively.
%
\subsection{The deterministic formulation}
%
In this subsection we consider a deterministic model of a disturbed system brought forward by Mortensen in \cite{Mor68}, see also \cite{Wil04}. It reads
\begin{equation}\label{eq: detModel}
\begin{alignedat}{2}
    x(t) &= A x(t) + B v(t)
    ~~~~~ &&t \in (0,T),\\
    x(0) &= x_0 + \eta, &&\\
    y(t) &= C x(t) + \mu(t)
    ~~~~~ &&t \in (0,T),
\end{alignedat}
\end{equation}
where the system matrices $A$, $B$, and $C$ are as in \eqref{eq: stochSys}. The disturbances in dynamics, initial value, and output are modeled by $v \in L^2(0,T;\mathbb{R}^m)$, $\eta \in \mathbb{R}^n$, and $\mu \in L^2(0,T;\mathbb{R}^r)$ and are assumed to be deterministic and unknown.
To increase readability for the remainder of this work we denote $\mathcal{L}_{t_1,t_2}^p = L^p(t_1,t_2;\mathbb{R}^d)$ and $\mathcal{H}_{t_1}^{t_2} = H^1(t_1,t_2;\mathbb{R}^n)$ for $0 \leq t_1 < t_2 \leq T $, $d \in \mathbb{N}$, and $0<p\leq\infty$ and 

The strategy for estimating $x(t)$, at $t \in (0,T]$ consists in identifying the energy minimal disturbances $\eta^*[t]$, $v^*[t]$, and $\mu^*[t]$ that fit the measured output $y(s)$, $0\leq s \leq t$, i.e., in solving
\begin{equation}\label{eq: energDist}
\begin{aligned}
    \min_{\eta \in \mathbb{R}^n, v \in \mathcal{L}_{0,t}^2, \mu \in \mathcal{L}_{0,t}^2}~ 
    &\Vert \eta \Vert_{\Gamma^{-1}}^2 
    + \int_0^t \Vert v(s) \Vert_{R^{-1}}^2 
    + \Vert \mu(s) \Vert_{Q^{-1}}^2 \, \mathrm{d}s,\\
    &\text{subject to}~
    \eqref{eq: detModel}~\text{on}~(0,t).
\end{aligned}
\end{equation}
This minimizer is used to solve \eqref{eq: detModel} for the associated trajectory $x^*[t]$ and the estimator is defined as $\widehat{x}(t) = x^*[t](t)$. Here and in the following for a positive definite matrix $M$ and vector $x$ of fitting size we denote $\Vert x \Vert_M^2 = x^\top M x$. 
The matrices $\Gamma$, $R$, and $Q$ correspond to the noise covariances from the stochastic formulation. Using their inverses as weighting matrices is a strategy also used in generalized least squares \cite{Ait36}.
In \cite{Wil04} Willems demonstrates that this estimator is given precisely by the Kalman filter equations known from the stochastic formulation. Indeed $\widehat{x}$ solves
\begin{equation}\label{eq: KF_fixed}
	\begin{aligned}
		\dot{\widehat{x}}(t) 
		&= A \widehat{x}(t)
		+ \Pi(t) C^\top Q^{-1} \left( y(t) - C \widehat{x}(t)\right)
        ~~~~~
		t \in (0,T) ,\\
		\widehat{x}(0) &= x_0,
	\end{aligned}
\end{equation}
where $\Pi $ is given as the unique solution of 
\begin{equation}\label{eq: Ricc_fixed}
	\begin{aligned}
		\dot{\Pi}(t) &= A \Pi(t) + \Pi(t) A^\top
		- \Pi(t) C^\top Q^{-1} C \Pi(t) 
		+ B R B^\top
		~~~
		~~~t \in (0,T),\\
		\Pi(0) &= \Gamma.
	\end{aligned}
\end{equation}

Our analysis follows the slightly different formulation of \cite{Mor68}.
By adding a final state $\xi$ and inserting the identities for $x(0)$ and $y$ Mortensen recasts \eqref{eq: energDist} into a problem of optimal control. It reads
\begin{alignat}{2}
    &\min_{x \in \mathcal{H}_0^t, v \in \mathcal{L}_{0,t}^2}~ J(x,v;t,\xi) &&= 
    \Vert x(0) - x_0 \Vert_{\Gamma^{-1}}^2 
    + \int_0^t \Vert v(s) \Vert_{R^{-1}}^2 
    + \Vert y - C x(s) \Vert_{Q^{-1}}^2 \, \mathrm{d}s,\label{eq: cost}\\
    &\text{subject to}
    ~~~~~~~~~~\dot{x}(s) &&= A x(s) + B v(s)
    ~~~~~ s \in (0,t), \label{eq: state}\\
    &~~~~~~~~~~~~~~~~~~~~~~~x(t) &&= \xi \label{eq: FV}
\end{alignat}
with associated value function
\begin{equation}\label{eq: VF}
\begin{aligned}
    \mathcal{V}(t,\xi) &= \inf_{x \in \mathcal{H}_0^t, v \in \mathcal{L}_{0,t}^2} J(x,v;t,\xi) ~~~\text{subject to}~ \eqref{eq: state}-\eqref{eq: FV} ~~~ t \in (0,T],\\
    \mathcal{V}(0,\xi) &= \Vert \xi - x_0 \Vert_{\Gamma^{-1}}^2.
\end{aligned}
\end{equation}
It represents the minimal amount of energy required to ensure that at time $t$ the system is in the state $\xi$. In \cite[Lem.~2.5]{KunSc25} it is shown that in the present case of linear dynamics the value function $\mathcal{V}$ is quadratic in space and given as
\begin{equation}\label{eq: VF_fixed}
    \mathcal{V}(t,\xi) = 
    (\xi - \widehat{x} (t))^\top \Pi^{-1}(t) (\xi - \widehat{x} (t))
    + \int_0^t \Vert y(s) - C \widehat{x} (s) \Vert_{Q^{-1}}^2 \, \mathrm{d}s,
\end{equation}
where $\widehat{x}$ and $\Pi$ are given via the Kalman filter equations \eqref{eq: KF_fixed} and \eqref{eq: Ricc_fixed}. Hence the estimator $\widehat{x}$ constructed via \eqref{eq: energDist} can equivalently be characterized as
\begin{equation*}
    \widehat{x}(t) \coloneqq \argmin_{\xi \in \mathbb{R}^n} \mathcal{V}(t,\xi).
\end{equation*}
We point out that the control formulation at hand is merely a tool to reconstruct energy minimal disturbances that most likely caused the measured output. 

Even though the following derivations and arguments are based on the deterministic perspective, we still refer to the solution $\Pi$ of \eqref{eq: Ricc_fixed} as the (error) covariance associated with the Kalman filter allowing for an interpretation of our results. We denote the inverse of the covariance as $P \coloneqq \Pi^{-1} $ and following \cite[Sec.~5.4]{DeGr70} we refer to it as the precision matrix. 
By multiplying \eqref{eq: Ricc_fixed} with $P(t)$ from the left and from the right we find that the precision matrix is characterized via the differential Riccati equation
\begin{equation}\label{eq: Prec_fixed}
\begin{aligned}
    \dot{P}(t) &= -A^\top P(t) - P(t) A - P(t) BRB^\top P(t) + C^\top Q^{-1} C,
    ~~~
	~~~t \in (0,T),\\
    P(0) &= \Gamma^{-1}.
\end{aligned}
\end{equation}
%
\section{Kalman filtering for uncertain systems}\label{sec: KF for uncertain systems}
%
The Kalman filter described in the previous section assumes exact knowledge of both the system dynamics represented by  $A$ and the nature of the occurring noise represented by the covariance  matrices $\Gamma$, $R$, and $Q$. In practical applications, however,  this information may not be available in exact form and the matrices may depend on parameters, \cite{Meh70,ShiJohMu07}. 
\subsection{Modeling uncertain parameters}\label[subsection]{subsec: ModUnc}
%
Following the description in \cite{KunSc25} the uncertainties of the system matrix $A$ are modeled as follows. 
Let $s_A \in \mathbb{N}$ and consider a matrix valued function
\begin{equation*}
    a \colon \mathbb{R}^{s_A} \rightarrow \mathbb{R}^{n,n}.
\end{equation*}
Further let 
$N_A \in \mathbb{N}$
and 
$\Sigma_A \subset \mathbb{R}^{s_A}$
be of cardinality 
$\vert \Sigma_A \vert = N_A$. Consider the discrete probability space $(\Sigma_A, \mathcal{P}(\Sigma_A),P_A)$
, where 
$\mathcal{P_A}(\Sigma_A)$
denotes the power set of $\Sigma_A$ and for $S \subset \Sigma_A$ define the probability function 
$P_A(S) = \frac{1}{N_A} \vert S \vert$.
Then the linear dynamics of the system are characterized by the matrix valued random variable 
\begin{equation*}
    \sigma \in \Sigma_A \mapsto a_\sigma = a(\sigma) \in \mathbb{R}^{n,n},
\end{equation*}
and $a_\sigma$ denotes the system matrix that realizes with the parameter $\sigma \in \Sigma_A$.
The uncertain weighting matrices $\gamma_\sigma$, $r_\sigma$, and $q_\sigma$ are defined  as random variables on $\Sigma_\Gamma$, $\Sigma_R$, and $\Sigma_Q$ with respective cardinalities and dimensions $N_\Gamma$, $s_\Gamma$, $N_R$, $s_R$, and $N_Q$, $s_Q$. The associated probabilities are defined as that for $A$. We denote $N = N_A \, N_\Gamma \, N_R \, N_Q$ and $s = s_A \, s_\Gamma \, s_R \, s_Q$. Then
\begin{equation*}
    \Sigma = \Sigma_A \times \Sigma_\Gamma \times \Sigma_R \times \Sigma_Q \subset \mathbb{R}^s
\end{equation*}
has cardinality $N$. Further we define the product space $(\Sigma, \mathcal{F},P)$, via
\begin{equation*}
\begin{aligned}
    \mathcal{F} &= \mathcal{P}(\Sigma_A) \otimes \mathcal{P}(\Sigma_\Gamma) \otimes \mathcal{P}(\Sigma_R) \otimes \mathcal{P}(\Sigma_Q) = \mathcal{P}(\Sigma),\\
    P(S_A \otimes S_\Gamma \otimes S_R \otimes S_Q) &= P_A(S_A) \, P_\Gamma(S_\Gamma) \, P_R(S_R)\, P_Q(S_Q) = \frac{1}{\vert S_A \otimes S_\Gamma \otimes S_R \otimes S_Q \vert}.
\end{aligned}
\end{equation*}
Finally the uncertain system and weighting matrices are defined as random variables on $(\Sigma,\mathcal{P}(\Sigma),P)$, where for $S \subset \Sigma$ we set $P(S) = \tfrac{1}{N} \vert S \vert $. Denote $\sigma = (\sigma_A,\sigma_\Gamma,\sigma_R,\sigma_Q) \in \Sigma$ and define
\begin{equation*}
\begin{alignedat}{2}
    \sigma\in \Sigma  \mapsto A_\sigma &
    =  a_{\sigma_A} \in\mathbb{R}^{n,n},
    \quad \sigma \in \Sigma  \mapsto \Gamma_\sigma 
    &&=  \gamma_{\sigma_\Gamma}\in \mathbb{R}^{n,n},\\
    \sigma\in \Sigma  \mapsto R_\sigma 
    &= r_{\sigma_R}\in \mathbb{R}^{m,m},
    \quad \sigma\in \Sigma  \mapsto Q_\sigma 
    &&= q_{\sigma_Q} \in \mathbb{R}^{r,r}.
\end{alignedat}
\end{equation*}
Note that  by construction they are mutually independent. 
For convenience we denote the $N$ elements of $\Sigma$ by $\sigma_k$, $k = 1,\dots,N$ and dependencies on $\sigma_k$ may be indicated by indexing $k$, e.g., $A_{\sigma_k} = A_k$. Also, for a given parameter $\sigma$ the associated four-tuple consisting of system matrix and covariance matrices is denoted by $\mathcal{S}_\sigma = (A_{{\sigma}},\Gamma_{{\sigma}}, R_{{\sigma}},Q_{{\sigma}}) \in \mathbb{R}^{n,n} \times \mathbb{R}^{n,n} \times \mathbb{R}^{m,m} \times \mathbb{R}^{r,r}$. To quantify the distance between such tuples for $p \in \mathbb{N}$ we introduce the norm
$
\Vert \mathcal{S}_{\sigma} \Vert_p^p
=
\Vert A_\sigma  \Vert_{\mathbb{R}^{n,n}}^p
+ \Vert \Gamma_\sigma \Vert_{\mathbb{R}^{n,n}}^p
+ \Vert R_\sigma \Vert_{\mathbb{R}^{m,m}}^p
+ \Vert Q_\sigma \Vert_{\mathbb{R}^{r,r}}^p
$. 

For any $\sigma_k \in \Sigma$ the corresponding Kalman filter is denoted by $\widehat{x}_k$ and the associated covariance and precision matrices are denoted by $\Pi_k$ and $P_k = \Pi_k^{-1}$, respectively. They are characterized via 
\begin{equation}\label{eq: KF_k}
	\begin{aligned}
		\dot{\widehat{x}}_k(t) 
		&= A_k \widehat{x}_k(t)
		+ \Pi_k(t) C^\top Q_k^{-1} \left( y(t) - C \widehat{x}_k(t)\right)
        ~~~~~
		t \in (0,T) ,\\
		\widehat{x}_k(0) &= x_0,
	\end{aligned}
\end{equation}
and
\begin{equation}\label{eq: Ricc_k}
	\begin{aligned}
		\dot{\Pi}_k(t) &= A_k \Pi_k(t) + \Pi_k(t) A_k^\top
		- \Pi_k(t) C^\top Q_k^{-1} C \Pi_k(t) 
		+ B R_k B^\top
		~~~
		~~~t \in (0,T),\\
		\Pi_k(0) &= \Gamma_k.
	\end{aligned}
\end{equation}
We note that the solution $\Pi_k$ is infinitely often continuously differentiable. Since the same holds for the mapping $M \mapsto M^{-1}$ defined on the set of symmetric positive definite matrices, we obtain that for all $k$ we have 
\begin{equation}\label{eq: P_smooth}
    P_k \in C^\infty([0,T];\mathbb{R}^{n,n}).
\end{equation}
To make our arguments more concise,  we denote for $k \in \{ 1,\dots,N \}$ and $t \in [0,T]$
\begin{equation*}
    r_k(t) = \int_0^t \Vert y(s) - C \widehat{x}_k(s) \Vert_{Q_k^{-1}}^2 \, \mathrm{d}s.
\end{equation*}
Analogous to \eqref{eq: VF_fixed} we obtain
\begin{equation}\label{eq: VF_k}
    \mathcal{V}_k(t,x) = \Vert x - \widehat{x}_k(t) \Vert_{P_k(t)}^2 + r_k(t).
\end{equation}
Next the estimators which will be investigated are described.
%
\subsection{Minimizing the expected energy}\label[subsection]{subs: MinExp}
%
We begin by presenting a design that was proposed and studied  in an earlier work, cf.~\cite{KunSc25}. 
For $t \in [0,T]$ consider the minimization problem
\begin{equation}\label{eq: minE}
    \min_{x \in \mathbb{R}^n}~ \mathbb{E} \left[ \mathcal{V}_\sigma (t,{x}) \right]
    = \min_{x \in \mathbb{R}^n}~\frac{1}{N} \sum_{k=1}^N \mathcal{V}_k (t,x).
\end{equation}
This minimization problem defines the first estimator. In \cite[Prop.~3.4]{KunSc25} we showed existence and uniqueness of the solution of \eqref{eq: minE} and throughout this work we denote the estimator defined as the pointwise solution of \eqref{eq: minE} as $\widehat{x}_0(t)$.
We further cite the following characterization.
\begin{proposition}\label[proposition]{prop: charEstMin}
    The state estimator $\widehat{x}_0$ is well-defined and given in terms of the Kalman filter trajectories and precision matrices associated with the individual Kalman filters. For $t \in [0,T]$ it holds
	\begin{equation}\label{eq: en_min}
		\widehat{x}_0(t) = \mathcal{P}^{-1}(t)
							\sum_{k=1}^N P_k(t) \, \widehat{x}_k(t), 
	\end{equation}
    where $\mathcal{P}(t) = \sum_{k=1}^N P_k(t) $.
\end{proposition}
This representation immediately implies a result on the regularity of the estimator.
\begin{corollary}\label[corollary]{cor: reg_Min_expEn}
    The estimator $\widehat{x}_0$ is square integrable and admits a square integrable weak derivative, i.e., $\widehat{x}_0 \in \mathcal{H}_0^T$. Its weak derivative is given by
    \begin{equation*}
        \dot{\widehat{x}}_0(t)
        =
        - \mathcal{P}^{-1}(t) \, \dot{\mathcal{P}}(t) \, \widehat{x}_0(t) 
        +
        \mathcal{P}^{-1}(t) \sum_{k=1}^N
            P_k(t) \dot{\widehat{x}}(t)
            + 
            \dot{P}_k(t) \widehat{x}_k(t).
    \end{equation*}
\end{corollary}
\begin{proof}
    The regularity of $P_k$ presented in \eqref{eq: P_smooth} implies $\mathcal{P}^{-1} \in C^\infty([0,T];\mathbb{R}^{n,n})$. For the solutions of \eqref{eq: KF_k} we have $\widehat{x}_k \in \mathcal{H}_0^T$. The claim for $\widehat{x}_0$ follows from the representation given in \eqref{eq: en_min}. The formula for the weak derivative is obtained by an application of the chain rule and \eqref{eq: en_min}.
\end{proof}
We proceed by introducing more risk averse designs.
%
\subsection{Minimizing the worst case energy}
%
For the construction of the second estimator consider the problem of minimizing the worst case energy 
\begin{equation}\label{eq: minMax}
    \min_{x\in \mathbb{R}^n}~
    \esssup_{\sigma \in \Sigma} \mathcal{V}_\sigma (t,x) 
    = 
    \min_{x\in \mathbb{R}^n}~\max_{k=1,\dots,N} \mathcal{V}_k(t,x).
\end{equation}
We begin by ensuring existence of a unique solution at any time. 
\begin{lemma}\label[lemma]{lem: Inft_Existence}
    For every $t \in [0,T]$ the problem \eqref{eq: minMax} admits a unique solution. 
\end{lemma}
\begin{proof}
  The value functions $\mathcal{V}_k$ defined in  \eqref{eq: VF_k} are strictly convex and coercive in $x$. Since the maximum  in \eqref{eq: minMax} is taken over finitely many indices $k$, the expression $\max_k \mathcal{V}_k(t,\cdot)$ inherits the strict convexity and coercivity properties. As a consequence, the existence of a unique minimizer then follows.
\end{proof}
This result enables the definition of the second estimator. 
\begin{definition}\label[definition]{def: minWorstCase}
    The state estimator minimizing the worst case energy at time $t \in [0,T]$ is denoted as $\widehat{x}_\infty(t)$ and given by the unique solution of the minimization problem \eqref{eq: minMax}.
\end{definition}
In the following we characterize the minimizer via a representation in terms of the family members $\widehat{x}_k$. This is done for every time point individually. Hence, from here on we fix a time $t \in [0,T]$.
To facilitate the discussion using tools from convex analysis we introduce the following notation. We define
\begin{equation*}
    f[t] \colon \mathbb{R}^n \to \mathbb{R},
    ~~~
    f[t](x) = \max_{k = 1,\dots,N} \mathcal{V}_k(t,x).
\end{equation*}
Now solving \eqref{eq: minMax} for the fixed time $t$ is equivalent to solving
\begin{equation}\label{eq: minmax_Fct}
    \min_{x \in \mathbb{R}^n} f[t](x).
\end{equation}
To formulate the first order optimality condition we introduce the set of active indices.
\begin{definition}
    For $x \in \mathbb{R}$ we define 
    \begin{equation*}
        I[t](x)
        = \{
        k \in \{ 1, \dots,N \}
        \colon
        \mathcal{V}_k(t,x) = f[t](x)
        \},
    \end{equation*}
    i.e., $I[t](x)$ is the set of all indices for which the maximum is attained. 
\end{definition}
With these definitions at hand we are able to characterize the minimizer $\widehat{x}_\infty(t)$.
\begin{theorem}\label[theorem]{thm: char_inft}
    Let $t \in [0,T]$ be arbitrary.
    There exists a subset $S(t) \subset I[t](\widehat{x}_\infty(t))$ with cardinality $\vert S(t) \vert = p \leq n+1$ and coefficients $\alpha(t) \in [0,1]^p$ such that it holds
    \begin{alignat*}{2}
        &(i)~~~~~
        &&\sum_{k\in S(t)} \alpha_k(t) = 1,\\
        &(ii)~~~~~
        && \left\{ P_k(t) ( \widehat{x}_\infty(t) - \widehat{x}_k(t) ) \right\}_{k \in S(t)}~\text{are affine independent}, \\
        &(iii)~~~~~
        &&\widehat{x}_\infty(t) 
        = \left( \sum_{k\in S(t)} \alpha_k(t) P_k(t) \right)^{-1} \sum_{k \in S} \alpha_k(t) P_k(t) \widehat{x}_k(t).
    \end{alignat*}
    In other words the estimator can be characterized in terms of a subset of $n+1$ or fewer family members in which the maximum is attained.
\end{theorem}
\begin{proof}
    Fix an arbitrary $t \in [0,T]$. Then $\widehat{x}_\infty(t)$ is the unique solution of \eqref{eq: minmax_Fct}. As discussed in the proof of \Cref{lem: Inft_Existence} $f[t]$ is strictly convex, hence $\widehat{x}_\infty(t)$ satisfies the optimality condition \cite[Def.~5.1.~and (5.5)]{EkeTem76} given by 
    \begin{equation*}
        0 \in \partial f[t](\widehat{x}_\infty(t)),
    \end{equation*}
    where $\partial f[t](\bar{x})$ denotes the subdifferential of $f[t]$ at $\bar{x}$. The subdifferential of a maximum is well known, cf.~\cite[Prop.~4.5.2 and Rem.~4.5.3]{Schi07}. Denoting the convex hull of a set $X$ by $\mathrm{conv}(X)$ we indeed have
    \begin{equation*}
    \begin{aligned}
        \partial f[t](\bar{x})
        &= 
        \mathrm{conv} \{ P_k(t) (\widehat{x}_\infty(t) - \widehat{x}_k (t) ) \colon k \in I[t](\widehat{x}_\infty(t)) \}\\
        &=
        \mathrm{conv} \{ P_k(t) (\widehat{x}_\infty(t) - \widehat{x}_k (t) ) \colon k \in S(t) \}, 
    \end{aligned}
    \end{equation*}
    where the second equality is justified by Caratheodory's theorem with a set $S(t) \subset I[t](\widehat{x}_\infty(t))$ such that $(ii)$ holds. The affine independence in particular implies $\vert S(t) \vert \leq n+1$. The optimality condition now ensures existence of $\alpha(t) \in [0,1]^p$ such that $(i)$ holds and
    \begin{equation*}
        0 = \sum_{k \in S(t)} \alpha_k(t) P_k(t) (\widehat{x}_\infty(t) - \widehat{x}_k(t)).
    \end{equation*}
    Rearranging terms yields $(iii)$ and the proof is complete.
\end{proof}
\begin{remark}
    This representation does not yield any time regularity of $\widehat{x}_\infty$. Since there is no straight forward way of characterizing the time dependence of the coefficients $\alpha_k$, at this point we can not even deduce measurability of the estimator. A result regarding this is obtained via a convergence argument in \Cref{cor: min_inf_essBd} below. 
\end{remark}
The two presented estimators $\widehat{x}_0$ and $\widehat{x}_\infty$ are the two extremes on the spectrum of risk aversion. While the minimizer of the expected value $\widehat{x}_0$ does not take risk into account at all, the worst-case minimizer $\widehat{x}_\infty$ is by construction the most risk averse. In the following we introduce a concept filling the gap in between these two extremes.

%
\subsection{Minimizing the entropic risk of the energy}
%
In this subsection we introduce another estimator as the minimizer of the so-called \textit{entropic risk measure}. To the best of the authors' knowledge it was introduced in the context of financial mathematics, cf.~\cite{FoeSchi16}. We commence with a brief presentation of the concept and its properties for the setting of this work. For a more general definition and detailed derivation we refer to \cite{FoeSchi16}.
This functional is also important in the context of machine learning, cf.~for example \cite{CalGaPo20}.  
\begin{definition}\label[definition]{def: entrRisk}
    Let $X\colon \Sigma \to\mathbb{R}$ be a random variable on the probability space $(\Sigma,\mathcal{P}(\Sigma),P)$ defined in \Cref{subsec: ModUnc}. The entropic risk measure of $X$ with risk aversion $\theta \in (0,\infty)$ is defined as
    \begin{equation*}
    	\rho_\theta(X) 
    	=
    	\frac{1}{\theta} \ln \left( \mathbb{E} \left[ e^{\theta X} \right] \right)
    	=
    	\frac{1}{\theta} \ln \left( \frac{1}{N} \sum_{k=1}^N e^{\theta X(\sigma_k)} \right).
    \end{equation*}
\end{definition}
We proceed by presenting some well known properties of the entropic risk measure. While we do not claim any novelty, for the sake of completeness we give a proof tailored to the setting of this work. 
\begin{lemma}\label[lemma]{lem: propEntrRisk}
    For all random variables $X$ on the probability space $(\Sigma,\mathcal{P}(\Sigma),P)$ it holds 
    \begin{equation*}
        \rho_\theta(X) \to 
        \esssup_{\sigma \in \Sigma} X(\sigma) = \max_{k=1,\dots,N} X(\sigma_k) 
        ~~~~~~\text{for}~\theta \to \infty. 
    \end{equation*}
\end{lemma}
\begin{proof}
    Let $X$ be such a fixed random variable and let $k^* \in \{ 1,\dots,N \}$ be such that $X(\sigma_{k^*}) = \max_{k = 1,\dots,N} X(\sigma_k) = \esssup_{\sigma \in \Sigma} X(\sigma)$. 
    Since the exponential is positive, for all $\theta \in (0,\infty)$ we have
    \begin{equation*}
        \frac{1}{N} \sum_{k=1}^N e^{\theta X(\sigma_k)} \geq \frac{1}{N} e^{\theta X(\sigma_{k^*})}.
    \end{equation*}
    Applying the logarithm on both sides and dividing by $\theta$ yields
    \begin{equation*}
        \rho_\theta(X) \geq X(\sigma_{k^*}) + \frac{\ln(\tfrac{1}{N})}{\theta}.
    \end{equation*}
    We hence find that for all $\theta \in (0,\infty)$ it holds
    \begin{equation*}
        X(\sigma_{k^*}) + \frac{\ln(\tfrac{1}{N})}{\theta}
        \leq
        \rho_{\theta}(X)
        \leq
        X(\sigma_{k^*}),
    \end{equation*}
    where the second estimate is justified by the monotonicity of the logarithm and the exponential. The assertion follows by considering the limit for $\theta \to \infty$.
\end{proof}

We now turn to the construction of the risk averse estimator. 
For its definition we consider the following minimization problem. For a fixed $t \in [0,T]$ consider
\begin{equation}\label{eq: minEntrRisk}
    \min_{x \in \mathbb{R}^n} \rho_\theta \left( \mathcal{V}_\sigma(t,x) \right)
    = 
    \min_{x \in \mathbb{R}^n}
    \frac{1}{\theta} 
    \ln \left(
        \frac{1}{N} \sum_{k=1}^N e^{\theta (\Vert x - \widehat{x}_k(t) \Vert_{P_k(t)}^2 + r_k(t))}
    \right).
\end{equation}
We proceed to formally define the third estimator as the minimizer of the entropic risk applied to the energy. 
\begin{definition}\label[definition]{def:minEntrR}
    Let $y \in \mathcal{L}_{0,T}^2$ and $t \in [0,T]$. The state estimator minimizing the entropic risk of the energy with risk aversion parameter $\theta \in (0,\infty)$ is denoted by $\widehat{x}_\theta(t)$ and defined as the solution of \eqref{eq: minEntrRisk}, assuming it exists and is unique.
\end{definition}
The following lemma shows that the estimator is well defined and characterizes it via an implicit equation. 
\begin{lemma}\label[lemma]{lem: charEntrMin}
    For every $t \in [0,T]$, $y \in \mathcal{L}_{0,T}^2$, and $\theta \in (0,\infty)$ there exists exactly one solution to the minimization problem \eqref{eq: minEntrRisk}. The solution, denoted by $\widehat{x}_\theta(t)$, satisfies
    \begin{equation}\label{eq: charEntrMin}
        \widehat{x}_\theta (t)
        =
        \left( \sum_{k=1}^N e^{\theta ( \Vert \widehat{x}_\theta (t) - \widehat{x}_k(t) \Vert_{P_k(t)}^2 + r_k(t) )} P_k(t) \right)^{-1}
        \sum_{k=1}^N e^{\theta ( \Vert \widehat{x}_\theta (t) - \widehat{x}_k(t) \Vert_{P_k(t)}^2 + r_k(t) )} P_k(t) \widehat{x}_k(t)
    \end{equation}
    and is the only element of $\mathbb{R}^n$ satisfying this relation.
\end{lemma}
\begin{proof}
    Let $t \in [0,T]$ be fixed. For the purpose of this proof we define $F\colon \mathbb{R}^n \to \mathbb{R}$ via
    \begin{equation*}
        F(x) \coloneqq
        \frac{1}{\theta} \ln \left( \frac{1}{N} \sum_{k=1}^N e^{\theta (\Vert x - \widehat{x}_k(t) \Vert_{P_k(t)}^2 + r_k(t))} \right).
    \end{equation*}
    In the following we show that $F$ is strictly convex. To that end note that it is smooth and the chain rule yields its gradient as
    \begin{equation}\label{eq: gradientEntrR}
        \nabla F(x)
        =
        2 
        \sum_{k=1}^N c_k[\theta,x](t)~ P_k(t) \widehat{x}_k(t).
    \end{equation}
    To characterize the Hessian we denote the outer product of a vector $x \in \mathbb{R}^n$ with itself via $[x] _\otimes^2 = x \otimes x = x x^\top \in \mathbb{R}^{n,n}$ and define 
    \begin{equation}\label{eq: constants_e}
        0
        \leq
        c_k[\theta,x](t) = \frac{e^{\theta ( \Vert x - \widehat{x}_k(t) \Vert_{P_k(t)}^2 + r_k(t) )}}
        {\sum_{j = 1}^N e^{\theta ( \Vert x - \widehat{x}_j(t) \Vert_{P_j(t)}^2 + r_j(t) )}}
        \leq 
        1.
    \end{equation}
    Note that for all $x$ and $t$ it holds $\sum_{k=1}^N c_k[\theta,x](t) = 1$. The chain rule then yields the Hessian as
    \begin{equation*}
    \begin{aligned}
        \nabla^2 F(x)
        &=
        2 \sum_{k=1}^N c_k[\theta,x](t) P_k(t)
        +
        4 \sum_{k=1}^N c_k[\theta,x](t)
        \left[P_k(t) (x - \widehat{x}_k(t))\right]_\otimes^2 \\
        &-
        4 \theta \left[ \sum_{k=1}^N c_k[\theta,x](t) P_k(t) \widehat{x}_k(t) \right]_\otimes^2.
    \end{aligned}
    \end{equation*}
    Now for any $z \in \mathbb{R}^n \backslash \{ 0 \}$ we find
    \begin{equation*}
    \begin{aligned}
        z^\top \nabla^2 F(x) z
        &= 
        2 \sum_{k=1}^N c_k[\theta,x](t)~ z^\top P_k(t) z
        +
        4 \theta \sum_{k=1}^N c_k[\theta,x](t) \Vert z^\top P_k(t) (x - \widehat{x}_k(t)) \Vert^2\\
        &-
        4 \theta \Vert \sum_{k=1}^N c_k[\theta,x](t) z^\top P_k(t) \widehat{x}_k(t)  \Vert^2.
    \end{aligned}
    \end{equation*}
    Due to the positive definiteness of $P_k$ the first summand is greater than zero. Since the $c_k$ are non negative and sum up to one, the convexity of the squared norm ensures that the sum of the last two terms is greater or equal than zero. Hence, the Hessian $\nabla^2 F(x)$ is positive definite implying strict convexity of $F$.
    One can quickly verify that $F$ also is coercive ensuring that it admits exactly one minimizer. Further, any element $\bar{x} \in \mathbb{R}^n$ minimizes $F$ if and only if $\nabla F (\bar{x}) = 0$. Setting the gradient given in \eqref{eq: gradientEntrR} to zero and performing standard calculus yields the claimed equation for $\widehat{x}_\infty(t)$.
\end{proof}
Note that an alternative to the characterization via \eqref{eq: charEntrMin} is given by 
\begin{equation}\label{eq: charMinEntrR2}
    \widehat{x}_\theta (t)
    =
    \left( \sum_{k=1}^N c_k[\theta,\widehat{x}_\theta(t)](t) P_k(t) \right)^{-1}
    \sum_{k=1}^N c_k[\theta,\widehat{x}_\theta(t)](t) P_k(t) \widehat{x}_k(t)
\end{equation}
with the coefficients given as in \eqref{eq: constants_e} satisfying $\sum_{k=1}^N c_k = 1$ and $c_k \in [0,1]$. 

We conclude the section by stating a result on the time regularity of $\widehat{x}_\theta$ depending on the regularity of the output $y$. The proof is presented in \Cref{subs: proof_reg_EntrR_Min}. 
\begin{theorem}\label[theorem]{thm: regularity_x_theta}
    Let $\theta \in (0,\infty)$.
    \begin{enumerate}
        \item Assuming $y \in \mathcal{L}_{0,T}^{2p}$ with an integer $1 \leq p <\infty$ it holds $\widehat{x}_\theta \in W^{1,p}(0,T;\mathbb{R}^n)$.
        \vspace{1mm}\\
        \item For $y \in C([0,T];\mathbb{R}^r)$ it holds $\widehat{x}_\theta \in C^1([0,T];\mathbb{R}^n)$.
    \end{enumerate}
    In both settings the (weak) derivative is given via the formula 
    \begin{equation}\label{eq: min_entrR_TimeDer}
        \dot{\widehat{x}}_\theta(t)
        =
        - M(t,\widehat{x}_\theta(t))^{-1}
        V(t,\widehat{x}_\theta(t)),
    \end{equation}
    where 
    \begin{equation}\label{eq: M_min_entrR_TimeDer}
    \begin{aligned}
        M(x,t)
        =
        \sum_{k=1}^N  e^{\theta \mathcal{V}_k(t,x)}
        \left(
        P_k(t) + 
        2 \theta P_k(t) (x - \widehat{x}_k(t)) \otimes (x - \widehat{x}_k(t)) P_k(t) 
        \right),
    \end{aligned}
    \end{equation}
    and
    \begin{equation}\label{eq: V_min_entrR_TimeDer}
    \begin{aligned}
        V(t,x) 
        &=  
        - \sum_{k=1}^N e^{\theta \mathcal{V}_k(t,x)} \,
        \left( P_k(t) \dot{\widehat{x}}_k(t)
        - \dot{P}_k(t) (x - \widehat{x}_k(t) ) \right)\\
        &+ \theta \sum_{k=1}^N e^{\theta \mathcal{V}_k(t,x)} \left(
        \Vert y(t) - C \widehat{x}_k(t) \Vert_{Q_k^{-1}}^2
        - 2 \langle x - \widehat{x}_k(t) , P_k(t) \dot{\widehat{x}}_k(t) \rangle
        \right)
        P_k(t) (x - \widehat{x}_k(t)).
    \end{aligned}
    \end{equation}
\end{theorem}
%
\section{Continuity of the state estimator in the risk aversion parameter}
%
In this section we analyze the behavior of the limit cases of the risk aversion parameter. We find that indeed, $\widehat{x}_\theta$ fills the gap between $\widehat{x}_0$ and $\widehat{x}_\infty$ in the sense that the estimators are continuous with respect to the risk aversion, in particular for the extreme cases.  
\begin{proposition}\label[proposition]{prop: pntwConv_in_theta}
    Let $t \in [0,T]$ be fixed. It holds
    \begin{equation*}
    \begin{alignedat}{2}
        \widehat{x}_\theta(t) &\to \widehat{x}_0(t), ~&&\theta \to 0,\\
        \widehat{x}_\theta(t) &\to \widehat{x}_\infty(t), ~~~&&\theta \to \infty,
    \end{alignedat}
    \end{equation*}
    where $\widehat{x}_0$ and $\widehat{x}_\infty$ are the estimators introduced in \Cref{subs: MinExp} and \Cref{def: minWorstCase}, respectively.
    Moreover $\theta\mapsto \widehat{x}_\theta(t)$ is continuous in $(0,\infty)$.
\end{proposition}
\begin{proof}
    To show the first limit we note that due to \Cref{lem: estWeightNorm} and \Cref{lem: UnifBoundTheta} in the appendix we have that $ \Vert \widehat{x}_\theta(t) - \widehat{x}_k(t) \Vert_{P_k(t)}^2 $ is bounded with respect to $\theta$. Since additionally $r_k[y](t)$ is independent of $\theta$, we find 
    \begin{equation*}
        \theta \left( \Vert \widehat{x}_\theta(t) - \widehat{x}_k(t) \Vert_{P_k(t)}^2 + r_k[y](t) \right)
        \to 0
        ~~~\text{for}~ \theta \to 0.
    \end{equation*}
    Since all involved operations are continuous, passing to the limit in \eqref{eq: charEntrMin} reveals
    \begin{equation*}
        \widehat{x}_\theta(t) 
        \to
        \left( \sum_{k=1}^N P_k(t) \right)^{-1} ~ \sum_{k=1}^N P_k(t)~ \widehat{x}_k (t)
        ~~~\text{for}~\theta \to 0
    \end{equation*}
    and the first claim follows with \Cref{prop: charEstMin}.

    To show the second convergence we utilize the characterization of $\widehat{x}_\theta(t)$ given in \eqref{eq: charMinEntrR2} which reads as
    \begin{equation}\label{eq: charMinEntrAlt}
        \widehat{x}_\theta(t)
        =
        \left( \sum_{k=1}^N c_k[\theta,\widehat{x}_\theta(t)](t) P_k(t) \right)^{-1}
        \sum_{k=1}^N c_k[\theta,\widehat{x}_\theta(t)](t) P_k(t) \widehat{x}_k(t),
    \end{equation}
    where the coefficients $0 \leq c_k[\theta,x](t) \leq 1$ are defined as in \eqref{eq: constants_e}.
    Now let $\theta_l$ be an arbitrary sequence such that $\theta_l \to \infty$ for $l \to \infty$. Since the $c_k[\theta_l,\widehat{x}_{\theta_l}(t)](t)$ are from a compact interval, any subsequence admits yet another subsequence $\theta_g$ such that for any $k$ we have that $c_k[\theta_g,\widehat{x}_{\theta_g}(t)](t) $ converges for $g \to \infty$. We denote the corresponding limits by $\beta_k(t)$. Since by construction for every $g \in \mathbb{N}$ and $t \in [0,T]$ the elements $\{c_k[\theta_g,\widehat{x}_{\theta_g}(t)](t) \}_{k=1}^N$ are non negative and sum up to one, the same holds for $\{ \beta_k(t) \}_{k=1}^N$. With \eqref{eq: charMinEntrAlt} it follows
    \begin{equation*}
        \lim_{g \to \infty}
        \widehat{x}_{\theta_g}(t)
        =
        \left( \sum_{k=1}^N \beta_k(t) P_k(t) \right)^{-1}
        \sum_{k=1}^N \beta_k(t) P_k(t) \widehat{x}_k(t) \eqqcolon \bar{x}.
    \end{equation*}
    Note that at this point $\bar{x}$ may depend on the choice of the subsequence $\theta_g$. 
    To remedy this, we now show that in fact $\bar{x} = \widehat{x}_\infty(t)$. To that end denote 
    \begin{equation*}
    F[\theta](x) 
    =
    \frac{1}{\theta} 
    \ln \left( \frac{1}{N} \sum_{k=1}^N e^{\theta (\Vert x - \widehat{x}_k(t) \Vert_{P_k(t)}^2 + r_k(t))} \right)
    ~~~~~
    \text{and}
    ~~~~~
    F_\infty(x) = \max_{k} \mathcal{V}_k(t,x)     
    \end{equation*}
    and consider
    \begin{equation}\label{eq: proofEst}
        \left\Vert
        F[\theta_g](\widehat{x}_{\theta_g}(t))
        -
        F_\infty(\bar{x}) 
        \right\Vert
        \leq
        \left\Vert
        F[\theta_g](\widehat{x}_{\theta_g}(t))
        -
        F[\theta_g](\bar{x})
        \right\Vert
        +
        \left\Vert 
        F[\theta_g](\bar{x})
        -
        F_\infty(\bar{x})
        \vphantom{
            F[\theta_g](\widehat{x}_{\theta_g}(t))
        }
        \right\Vert.
    \end{equation}
    Utilizing \Cref{lem: propEntrRisk} we find that the second summand on the right hand side converges to zero. To show convergence of the first summand on the right hand side we first estimate the gradient $\Vert \nabla F[\theta](x) \Vert$ for some $x \in \mathbb{R}^n$ and $\theta \in (0,\infty)$. We have
    \begin{equation*}
        \Vert \nabla F[\theta](x) \Vert
        = 2
        \left\Vert 
            \sum_{k=1}^N c_k[\theta,x](t) P_k(t) \widehat{x}_k(t)
        \right\Vert,
    \end{equation*}
    where the coefficients $c_k[\theta,x]$ are defined as above and hence bounded by one. Due to the continuity of $P_k$ and $\widehat{x}_k$ there exists a constant $L>0$ independent of $t$, $\theta$, $x$, and $k$ such that 
    \begin{equation*}
        \Vert \nabla F[\theta](x) \Vert \leq L.
    \end{equation*}
    We now apply Taylor's theorem \cite[Thm.~4.C]{Zei95AMS109} and obtain 
    \begin{equation*}
    \begin{aligned}
        \left\Vert 
        F[\theta_g](\widehat{x}_{\theta_g}(t))
        -
        F[\theta_g](\bar{x})
        \right\Vert
        &\leq
        \int_0^1 \Vert \nabla F[\theta_g](\tau \widehat{x}_{\theta_g}(t) + (1-\tau) \bar{x} ) \Vert \, \mathrm{d}\tau
        ~
        \Vert \widehat{x}_{\theta_g}(t) - \bar{x} \Vert\\
        &\leq
        L \Vert \widehat{x}_{\theta_g}(t) - \bar{x} \Vert,
    \end{aligned}
    \end{equation*}
    where the convergence of $\widehat{x}_{\theta_g}(t)$ implies convergence to zero of the right hand side.
    We have therefore shown that the right hand side of \eqref{eq: proofEst} converges to zero for $g \to \infty$, implying 
    \begin{equation*}
        \lim_{g \to \infty}F[\theta_g](\widehat{x}_{\theta_g}(t))
        =
        F_\infty(\bar{x}).
    \end{equation*}
    As by construction $\widehat{x}_{\theta_g}$ minimizes $F[\theta_g]$ we have
    \begin{equation*}
        F[\theta_g](\widehat{x}_{\theta_g})
        \leq 
        F[\theta_g](\widehat{x}_\infty(t)).
    \end{equation*}
    Applying again \Cref{lem: propEntrRisk} we pass to the limit on both sides and obtain
    \begin{equation*}
        F_\infty(\bar{x})
        \leq
        F_\infty(\widehat{x}_\infty(t)).
    \end{equation*}
    Since $\widehat{x}_\infty(t)$ is the unique minimizer of $F_\infty$, it follows $\bar{x} = \widehat{x}_\infty(t)$.
    We have now shown that any subsequence of $\widehat{x}_{\theta_l}(t)$ admits a subsequence $\widehat{x}_{\theta_g}(t)$ that converges to $\widehat{x}_\infty(t)$ and the second convergence claim is shown.

    It remains to show the continuity for $\theta \in (0,\infty)$.
    This follows from the fact that for each $t\in [0,T]$ this mapping is uniformly bounded with respect to $\theta$, see \Cref{lem: UnifBoundTheta}. This allows to extract converging subsequences for families $\theta\to \widehat{x}_\theta(t)$ with $\theta \to \bar\theta$  for some $\bar\theta \in (0, \infty)$. Subsequently one can pass to the limit in \eqref{eq: charEntrMin} to assert that this limit is necessarily $\widehat{x}_{\bar\theta}(t)$. 
\end{proof}
The arguments made in the proof further yield a representation of $\widehat{x}_\infty(t)$ with coefficients that are measurable in time. We obtain the following result.
\begin{corollary}\label[corollary]{cor: min_inf_essBd}
	The worst case minimizer $\widehat{x}_\infty$ is measurable  and essentially bounded, and hence $\widehat{x}_\infty \in \mathcal{L}_{0,T}^\infty$.
\end{corollary}
\begin{proof}
	Above it was shown that it holds
	\begin{equation*}
		\widehat{x}_\infty(t) 
		=
		\left( \sum_{k=1}^N \beta_k(t) P_k(t) \right)^{-1}
		\sum_{k=1}^N \beta_k(t) P_k(t) \widehat{x}_k(t),
	\end{equation*}
	where $\beta_k(t)$ are the pointwise limits of $c_k[\theta_g,\widehat{x}_{\theta_g}(t)](t)$ as defined in \eqref{eq: constants_e}. Since the $c_k[\theta_g,\widehat{x}_{\theta,}(t)]$ are measurable, so are the $\beta_k(t)$. As discussed in the proof of \Cref{prop: pntwConv_in_theta} we have $0 \leq \beta_k(t) \leq 1$. Moreover $P_k$ and $\widehat{x}_k$ are continuous in $t \in [0,T]$. Consequently $\sum_{k=1}^N \beta_k(t) P_k(t) \widehat{x}_k(t)$ and $\sum_{k=1}^N \beta_k(t) P_k(t)$ are measurable. 
    According to \Cref{lem: InvOfPrecSumBD} $\sum_{k=1}^N \beta_k(t) P_k(t)$ is invertible with an inverse bounded uniformly in $t$. The continuity of $M \mapsto M^{-1}$ on positive definite matrices implies measurability of $\sum_{k=1}^N \beta_k(t) P_k(t)$. 
    These facts imply the measurability and uniform boundedness of $\hat x_\infty.$ 
\end{proof}
\begin{remark}
	Using the affine independence established in \Cref{thm: char_inft} one could show a relation between the here discussed representation via $\beta_k$ and the one via $\alpha_k$ as presented in \Cref{thm: char_inft}. Since this does not offer any further insight towards the properties of $\widehat{x}_\infty$, we omit the required technical arguments. 
\end{remark}
Due to the $\theta$-uniform bound on $\widehat{x}_\theta(t)$ the pointwise convergence implies convergence in appropriate function spaces.
\begin{proposition}
    Let $y \in \mathcal{L}_{0,T}^2$ and $p \in \mathbb{N}$ such that $1 \leq p <\infty$. It holds 
    \begin{equation*}
    \begin{alignedat}{2}
        \Vert \widehat{x}_\theta - \widehat{x}_0 \Vert_{W^{1,1}(0,T;\mathbb{R}^n)} &\to 0
        ~~~~~
        &&\text{for}~ \theta \to 0,\\
        \Vert \widehat{x}_\theta - \widehat{x}_\infty \Vert_{\mathcal{L}_{0,T}^p} &\to 0
        ~~~~~
        &&\text{for}~ \theta \to \infty.
    \end{alignedat}
    \end{equation*}
    %
\end{proposition}
\begin{proof}
    We first consider the limit for $\theta \to \infty$. The pointwise convergence established in \Cref{prop: pntwConv_in_theta} together with the $\theta$-uniform bound from \Cref{lem: UnifBoundTheta} allow for an application of the dominated convergence theorem yielding the asserted convergence in $\mathcal{L}_{0,T}^p$ for any $1\leq p<\infty$. 

    The exact same argument yields convergence of $\widehat{x}_\theta$ to $\widehat{x}_0$ in $\mathcal{L}_{0,T}^1$ for $\theta \to 0$. It remains to show $L^1$ convergence of the derivatives. According to \Cref{thm: regularity_x_theta} for all $\theta \in (0,\infty)$ the weak derivative of $\widehat{x}_\theta$ exists, lies in $\mathcal{L}_{0,T}^1$ and for almost all $t\in[0,T]$ is given by the formula \eqref{eq: min_entrR_TimeDer}. Passing to the limit $\theta \to 0$ in said formula and comparing the outcome with the formula given in \Cref{cor: reg_Min_expEn} we find that we have pointwise convergence almost everywhere, i.e., for almost all $t \in[0,T]$ it holds
    \begin{equation*}
        \dot{\widehat{x}}_\theta(t) 
        \to 
        \dot{\widehat{x}}_0(t)
        ~~~~~
        \text{for}~
        \theta \to 0.
    \end{equation*}
    It remains to establish an integrable upper bound that dominates $\dot{\widehat{x}}_\theta(t)$ almost everywhere. Denoting the uniform bound for $\widehat{x}_\theta(t)$ established in \Cref{lem: UnifBoundTheta} by $C = \hat{C} (\Vert x_0 \Vert^2 + \Vert y \Vert_{\mathcal{L}_{0,T}^2})^2$ and utilizing \eqref{eq: min_entrR_TimeDer} we find
    \begin{equation*}
        \Vert \dot{\widehat{x}}_\theta(t) \Vert
        \leq
        \Vert M(\widehat{x}_\theta(t),t)^{-1} \Vert_2
        \Vert V(\widehat{x}_\theta(t),t) \Vert,
    \end{equation*}
    with $M$ and $V$ as in \eqref{eq: M_min_entrR_TimeDer} and \eqref{eq: V_min_entrR_TimeDer}, respectively. We begin by estimating the first factor. In the following, for two symmetric matrices $M_1,M_2 \in \mathbb{R}^{n\times n}$ we write $M_1 \geq M_2$ whenever $M_1 - M_2$ is positive semi-definite.  From the definition of $M$ we deduce that
    \begin{equation*}
        M(\widehat{x}_\theta(t),t)
        \geq 
        \sum_{k=1}^N P_k(t)
        \geq P_1(t).
    \end{equation*}
    It follows 
    \begin{equation}
        \Vert M(\widehat{x}_\theta(t),t)^{-1} \Vert_2 
        \leq 
        \Vert P_1(t) \Vert_2
        \leq
        \max_{s \in [0,T]} \Vert P_1(s) \Vert_2 = c_1
    \end{equation}
    with $c_1$ independent of $t$ and $\theta$. Turning to  $V(\widehat{x}_\theta(t),t)$ the only critical term is 
      \begin{equation}\label{eq:aux1}
        \theta \sum_{k=1}^N e^{\Vert \widehat{x}_\theta(t) - \widehat{x}_k(t) \Vert_{P_k(t)}^2 + r_k(t) }
        \Vert y(t) - C \widehat{x}_k(t) \Vert^2 P_k(t) (\widehat{x}_\theta(t) - \widehat{x}_k(t)).
    \end{equation}
    The remaining ones are either continuous in $t \in [0,T]$ or they can be uniformly estimated using \Cref{lem: UnifBoundTheta}. Note  in particular that $\theta$ appearing as factor does not pose any issue, as we are considering the limit for $\theta \to 0$. Turning to \eqref{eq:aux1}
   we use the known bounds for $P_k$, $\widehat{x}_k$ and $\widehat{x}_\theta(t)$,  so that since $y\in \mathcal{L}^2_{0,T}$ this term can be bounded by an $L^1$ function. 

    We have now shown that $\dot{\widehat{x}}_\theta $ converges pointwise almost everywhere and is pointwise almost everywhere bounded by an integrable function. 
    Consequently the theorem of dominated convergence yields 
    \begin{equation*}
        \dot{\widehat{x}}_\theta \to \dot{\widehat{x}}_0
        ~~~~~
        \text{for}~
        \theta \to 0,
        ~~~~~~~~~~
        \text{in}~ \mathcal{L}_{0,T}^1.
    \end{equation*}
    Together with the $L^1$ convergence of $\widehat{x}_\theta$ to $\widehat{x}_0$ this yields the asserted convergence of $\widehat{x}_\theta$ in $W^{1,1}$.
\end{proof}
%

%
\section{Error analysis}
%
Throughout this section we denote the true, hidden parameter as $\bar{\sigma} \in \Sigma$. We establish a priori bounds for the weighted distance of each estimator to the Kalman filter $\widehat{x}_{\bar{\sigma}}$ associated with the true parameter characterized via
\begin{equation}\label{eq: KF_bar}
	\begin{aligned}
		\dot{\widehat{x}}_{\bar{\sigma}}(t) 
		&= A_{\bar{\sigma}} \widehat{x}_{\bar{\sigma}}(t)
		+ \Pi_{\bar{\sigma}}(t) C^\top Q_{\bar{\sigma}}^{-1} \left( y(t) - C \widehat{x}_{\bar{\sigma}}(t)\right)
        ~~~~~
		t \in (0,T) ,\\
		\widehat{x}_{\bar{\sigma}}(0) &= x_0,
	\end{aligned}
\end{equation}
and
\begin{equation}\label{eq: Ricc_bar}
	\begin{aligned}
		\dot{\Pi}_{\bar{\sigma}}(t) &= A_{\bar{\sigma}} \Pi_{\bar{\sigma}}(t) + \Pi_{\bar{\sigma}}(t) A_{\bar{\sigma}}^\top
		- \Pi_{\bar{\sigma}}(t) C^\top Q_{\bar{\sigma}}^{-1} C \Pi_{\bar{\sigma}}(t) 
		+ B R_{\bar{\sigma}} B^\top
		~~~
		~~~t \in (0,T),\\
		\Pi_{\bar{\sigma}}(0) &= \Gamma_{\bar{\sigma}}.
	\end{aligned}
\end{equation}
We begin by citing two results from an earlier work. The following auxiliary result is required for proving error bounds for the estimators. The proof is presented in \cite[Lem.~4.6]{KunSc25}.
\begin{lemma}\label[lemma]{lem: est_Kal}
    Let $\widehat{x}_k$, $\Pi_k$, $\widehat{x}_{\bar{\sigma}}$, and $\Pi_{\bar{\sigma}}$ be the solutions of \eqref{eq: KF_k}, \eqref{eq: Ricc_k}, \eqref{eq: KF_bar}, and \eqref{eq: Ricc_bar}, respectively. Then there exists a constant $c_k > 0$ independent of $t\in[0,T]$ such that
    \begin{equation*}
    \begin{aligned}
        \Vert \widehat{x}_k(t) - \widehat{x}_{\bar{\sigma}}(t) \Vert
        \leq
        c_k 
        \Vert \mathcal{S}_k - \mathcal{S}_{\bar{\sigma}} \Vert_1.
    \end{aligned}
    \end{equation*} 
\end{lemma}
With this technical result at hand the following bound can be shown for the minimizer of the expected energy $\widehat{x}_0$ as defined in \Cref{subs: MinExp}, for a proof see \cite[Prop.~4.10]{KunSc25}.
\begin{proposition}\label[proposition]{prop: errBd_expectedEn}
	There exists a constant $c_0>0$ independent of $N$ such that for all $t \in [0,T]$ it holds
	\begin{equation*}
		\begin{aligned}
			\Vert \widehat{x}_0(t) - \widehat{x}_{\bar{\sigma}}(t) \Vert_{P_{\bar{\sigma}(t)}}
			\leq
			c_0~
			\mathbb{E} \left[
			\Vert \mathcal{S}_\sigma - \mathcal{S}_{\bar{\sigma}} \Vert_1
			\right].
		\end{aligned}
	\end{equation*}
\end{proposition}
We proceed to show similar bounds for the remaining two estimators.
\begin{proposition}\label[proposition]{prop: errEst_Inf}
    There exists a constant $c_\infty>0$ independent of $N$ such that for all $t \in [0,T]$ and $\theta \in (0,\infty)$ it holds
    \begin{equation*}
    \begin{aligned}
        \Vert \widehat{x}_\infty(t) - \widehat{x}_{\bar{\sigma}}(t) \Vert_{P_{\bar{\sigma}(t)}}
        &\leq
        c_\infty~
        \max_{j \in \{ 1,\dots,N \}}
        \Vert \mathcal{S}_{j} - \mathcal{S}_{\bar{\sigma}} \Vert_1,\\
        \Vert \widehat{x}_\theta(t) - \widehat{x}_{\bar{\sigma}}(t) \Vert_{P_{\bar{\sigma}(t)}}
        &\leq
        c_\infty~
        \max_{j \in \{ 1,\dots,N \}}
        \Vert \mathcal{S}_{j} - \mathcal{S}_{\bar{\sigma}} \Vert_1.
    \end{aligned}
    \end{equation*}
\end{proposition}
\begin{proof}
    We begin by showing the first estimate.
    The proof is based on the representation of $\widehat{x}_\infty(t)$ given in \Cref{thm: char_inft}. Let $\alpha_k(t)$ and $S(t)$ be as described there. For convenience we fix $t$ and denote $\alpha_k = \alpha_k(t)$ and $S = S(t)$. 
    With \Cref{lem: estWeightNorm}, \Cref{lem: est_Kal}, and \Cref{lem: InvOfPrecSumBD} we obtain
    \begin{equation*}
    \begin{aligned}
        \Vert \widehat{x}_\infty(t) - \widehat{x}_{\bar{\sigma}}(t) \Vert_{P_{\bar{\sigma}}(t)}
        &=
        \left\Vert 
        ( \sum_{k \in S} \alpha_k P_k(t) )^{-1}
        \sum_{k \in S} \alpha_k P_k(t) (\widehat{x}_k(t) - \widehat{x}_{\bar{\sigma}}(t))
        \right\Vert_{P_{\bar{\sigma}}(t)}\\
        &\leq
        \lambda_\mathrm{max}^{\tfrac{3}{2}}
        \lambda_\mathrm{min}^{-1} 
        \sum_{k \in S} \alpha_k
        \Vert \widehat{x}_k(t) - \widehat{x}_{\bar{\sigma}}(t) \Vert
        \leq
        \lambda_\mathrm{max}^{\tfrac{3}{2}}
        \lambda_\mathrm{min}^{-1} 
        \sum_{k \in S} \alpha_k
        c_k \Vert \mathcal{S}_k - \mathcal{S}_{\bar{\sigma}} \Vert_1.
    \end{aligned}
    \end{equation*}
    The assertion follows by setting $c_\infty = \lambda_\mathrm{max}^{\tfrac{3}{2}}
        \lambda_\mathrm{min}^{-1}  \max_{k \in S} c_k$ and estimating $\Vert \mathcal{S}_k - \mathcal{S}_{\bar{\sigma}} \Vert_1$ from above by the maximum over $k \in S$.

    Using the characterization of $\widehat{x}_\theta$ given in \eqref{eq: charMinEntrR2} the proof for the second estimate can be carried out analogously.
\end{proof}
We conclude the section by presenting another error bound for the minimizer of the entropic risk $\widehat{x}_\theta$ as defined in \Cref{def:minEntrR}. 
\begin{proposition}
    Let $c_\infty$ be the constant from \Cref{prop: errEst_Inf}. For all $\theta \in (0,\infty)$ and $t \in [0,T]$ it holds
    \begin{equation*}
        \Vert \widehat{x}_\theta(t) - \widehat{x}_{\bar{\sigma}}(t) \Vert_{P_{\bar{\sigma}(t)}}
        \leq 
        c_\infty \,
        e^{\theta J(t,\theta) } \,
        \mathbb{E} \left[
			\Vert \mathcal{S}_\sigma - \mathcal{S}_{\bar{\sigma}} \Vert_1
			\right],
    \end{equation*}
    where $J(t,\theta) = \max_{k,j \in \{ 1,\dots,N \}} 
    \left[ 
    \mathcal{V}_k(t,\widehat{x}_\theta(t)) - \mathcal{V}_j(t,\widehat{x}_\theta(t)) 
    \right]$. 
\end{proposition}
\begin{proof}
    Note first that with \Cref{lem: InvOfPrecSumBD} we obtain
    \begin{equation*}
        \Vert (\sum_{k=1}^N P_k(t) )^{-1} \Vert_2
        = \frac{1}{N}  \Vert (\sum_{k=1}^N \frac{1}{N} P_k(t) )^{-1} \Vert_2
        \leq (N \lambda_\mathrm{min})^{-1}.
    \end{equation*}
    Now utilizing \eqref{eq: charEntrMin} and \Cref{lem: estWeightNorm} we find
    \begin{equation*}
    \begin{aligned}
        \Vert \widehat{x}_\theta(t) - \widehat{x}_{\bar{\sigma}}(t) \Vert_{P_{\bar{\sigma}(t)}}
        &\leq 
        \lambda_\mathrm{max}^{\tfrac{3}{2}}
        \Vert ( \sum_{k=1}^N e^{\theta \mathcal{V}_k(t,\widehat{x}_\theta(t))} P_k(t) )^{-1} \Vert_2
        \sum_{k=1}^N e^{\theta\mathcal{V}_k(t,\widehat{x}_\theta(t))}
        \Vert \widehat{x}_k(t) - \widehat{x}_{\bar{\sigma}}(t) \Vert\\
        &\leq
        \lambda_\mathrm{min}^{-1} \lambda_\mathrm{max}^{\tfrac{3}{2}} \,
        e^{ \theta (\mathcal{V}_{k_\mathrm{max}}(t,\widehat{x}_\theta(t))
        -\mathcal{V}_{k_\mathrm{min}}(t,\widehat{x}_\theta(t)))}
        \frac{1}{N} \sum_{k=1}^N \Vert \widehat{x}_k(t) - \widehat{x}_{\bar{\sigma}}(t) \Vert,
    \end{aligned}
    \end{equation*}
    where $k_\mathrm{min}$ is chosen such that $\min_k \mathcal{V}_k(t,\widehat{x}_\theta(t)) = \mathcal{V}_{k_\mathrm{min}}(t,\widehat{x}_\theta(t)) $ and $k_\mathrm{max}$ is chosen analogously. Applying \Cref{lem: est_Kal} yields the assertion.
\end{proof}
%
\section{Numerical experiments}
%

This section presents numerical results illustrating and comparing the proposed estimators using two examples. 

%
\subsection{General setup}\label[subsection]{subs: genSetup}
%
We consider disturbed uncertain linear systems of the form 
\begin{equation}\label{eq: numSys}
\begin{alignedat}{2}
    \dot{x}(t) &= A_\sigma x(t) + B v(t) 
    ~~~~~&& t \in (0,T),
    ~~~~~~~~x(0) = x_0 + \eta,\\
    y(t) &= C x(t) + \mu(t)
    ~~~~~&& t \in (0,T),
\end{alignedat}
\end{equation}
where $\sigma \in \Sigma_A \subset \mathbb{R}^{s_A}$ with $ \vert \Sigma_A \vert = N_A$, as in \Cref{subsec: ModUnc}. The errors are associated with positive definite covariance matrices $\Gamma$, $R$, and $Q$ independent of $\sigma$. 

All ODEs are solved using the \matlab\, solver \texttt{ode15s} with a relative tolerance of $10^{-8}$ and an equidistant time grid $0 = t_0 < t_1 < \dots < t_{1000} = T$.
Kalman filter equations such as \eqref{eq: KF_fixed}-\eqref{eq: Ricc_fixed} are transformed into $n+n^2$ dimensional ODEs. 

To generate outputs $y \in \mathcal{L}_{0,T}^2$ we utilize the \matlab\, function \texttt{normrnd} to construct realizations of $\eta \sim \mathcal{N}(0,\Gamma)$, $v(t_k) \sim \mathcal{N}(0,R)$, and $\mu(t_k) \sim \mathcal{N}(0,Q)$ for $k = 0,\dots,1000$. The time continuous errors $v \in \mathcal{L}_{0,T}^2$ and $\mu \in \mathcal{L}_{0,T}^2$ are obtained via linear interpolation. After choosing a $\bar{\sigma} \in \Sigma_A$ to be the true, hidden parameter $y$ is obtained by solving \eqref{eq: numSys} with the constructed errors and $\bar{\sigma}$.
Note that this construction of $v$ and $\mu$ only works in discrete time and the extension via interpolation is of heuristic nature and does not converge to any function for $\Delta t \to 0$. In spirit this limit leads to the stochastic formulation via random processes and SDEs. 

To construct the estimators one needs to solve for the individual Kalman filters $ \widehat{x}_k $ and $P_k$, $k=1,\dots,N_A$. As they are independent of each other this is done in parallel using \texttt{parfor} from the \matlab\, Parallel Computing Toolbox.

With the $\widehat{x}_k$ and $P_k$ at hand, the risk neutral estimator $\widehat{x}_0$ is realized directly via \eqref{eq: en_min}. The realization of the risk averse estimators $\widehat{x}_\theta$, $\theta \in (0,\infty)$, however, requires a more elaborate approach. We opt for a gradient descent scheme to approximate $\widehat{x}_\theta(t_k)$, $k=0,\dots,1000$ as the minimizer of $\rho_\theta(\mathcal{V}_\sigma(t_k,x))$, cf.~\eqref{eq: minEntrRisk}. For each iteration the stepsize is constructed using the Barzilai-Borwein method \cite{AzmK20,BarB88} combined with Armijo backtracking. Denoting the risk aversion parameters of interest by $\theta_1,\dots,\theta_l$, sorted increasingly, we initialize the gradient descent for $\theta_1$ with $\widehat{x}_0$ and in an iterative manner the construction of $\widehat{x}_{\theta_g}$ is initialized by 
$\widehat{x}_{\theta_{g-1}}$.
For the evaluation of $\rho_\theta(\mathcal{V}_\sigma)$ and its gradient we employ the well known log-sum-exp trick to avoid overflow. 
The realizations were implemented in \matlab\, R2024b and the code is available in \cite{Sch25}.

%
\subsection{Harmonic oscillator with uncertain damping}
%
As a first example we consider a harmonic oscillator which in first order form is modeled as 
\begin{equation*}
\begin{alignedat}{2}
    \dot{x}(t) &= 
    \begin{bmatrix}
        0 & 1 \\
        - \frac{k}{m} & - \frac{c}{m}
    \end{bmatrix}
    x(t)
    +
    \begin{bmatrix} 0 \\ 1 \end{bmatrix} v(t)
    ~~~ &&t \in (0,T),\\
    x(0) &= x_0 + \eta,&&\\ 
    y(t) & = \begin{bmatrix} 1 & 0 \end{bmatrix} x(t)
    + \mu(t)
    ~~~ &&t \in (0,T).
\end{alignedat}
\end{equation*}
The two components of the state vector $x = \begin{bmatrix} x_1 & x_2 \end{bmatrix}^\top$ represent the position and velocity of the system.
Further $m>0$, $c \geq 0$, and $k \geq 0$ represent the mass, damping coefficient, and spring constant, respectively. The function $v \in \mathcal{L}_{0,T}^2$ represents the unknown disturbance in the dynamics. The modeled initial position and velocity are given by $x_{0,1}$ and $x_{0,2}$ and are subject to the initial errors $\eta_1$ and $\eta_2$. Finally, $y$ denotes the measurement of the position 
with a deterministic but unknown output error $\mu \in \mathcal{L}_{0,T}^2$.

For our experiments we fix the mass and the spring constant to one, i.e., $m=k=1$, the time horizon is set to $T=5$. As an undisturbed initial state we use 
$x_0 = \begin{bmatrix} 1 & 0 \end{bmatrix}^\top$ and for the covariances we set $\Gamma = 0.1 ~\mathrm{Id} $, $R = 0.05$, and $Q = 0.05$. In the following we present results for two sets of possible damping parameters each containing $N_A = 100$ elements. 

The first parameter set is given via uniform samples from the interval $[0.1,3]$, cf.~\Cref{subf: parameters_1}, the second parameter set is sampled via a log-normal distribution with mean $-0.25$ and variance $0.5$, both in the logarithmic scale, illustrated in \Cref{subf: parameters_2}. For both sets the parameter of maximal value is designated as the true underlying parameter $\bar{\sigma}$. 
The resulting disturbed state $x$, and associated state estimators are presented in \Cref{subf: phase_1} and \Cref{subf: phase_2}, respectively. 
We observe that the Kalman filter $\widehat{x}_{\bar{\sigma}}$ associated with the hidden parameter yields a satisfying reconstruction of the hidden state $x$ validating the Kalman based approach. However, since $\bar{\sigma}$ is not known, $\widehat{x}_{\bar{\sigma}}$ can not be constructed and we turn to the uncertainty based estimators. 

Comparing the risk neutral estimator $\widehat{x}_0$ with the risk averse estimators $\widehat{x}_\theta$, for $\theta = 0.1, 0.5, 1, 20, 750, 1000$ we find that
if the underlying true parameter is an outlier 
then the risk averse approach yields noticeably more accurate reconstructions than the risk neutral one, cf.~\Cref{subf: phase_2}. In the case of uniformly distributed parameters we observe
relatively large similarities between the risk neutral and risk averse approaches.
In both phase plots, see \Cref{subf: phase_1} and \Cref{subf: phase_2}, it appears that convergence for $\theta \to \infty$ has almost been reached. Indeed, the difference of the trajectories $\widehat{x}_{750}$ and $\widehat{x}_{1000}$ is close to zero.

A deeper analysis of the numerical results is offered in \Cref{tab: osc1} where we present the time integral over the entropic risk evaluated along a selection of state estimates. More specifically for $\theta = 0,0.5,20,1000$ and $\tau = 0,0.5,20,1000,\infty$ we present the values
\begin{equation*}
    \int_0^T \rho_\tau (\mathcal{V}_\sigma (t,\widehat{x}_\theta(t))) \, \mathrm{d}t,
\end{equation*}
where for $\tau \in (0,\infty)$ we denote by $\rho_\tau$ the entropic risk as defined in \Cref{def: entrRisk}, and $\rho_0$ and $\rho_\infty$ denote the expectation and essential supremum, respectively. By construction it holds that $\widehat{x}_\theta(t)$ minimizes $\rho_\theta(\mathcal{V}_\sigma(t,\cdot))$. Consequently, for both parameter sets the minima in rows one to four (printed in bold face) lie on the diagonal. Additionally, the largest risk aversion parameter $\theta = 1000$ leads to the smallest value of the essential supremum, reiterating the proximity of $\widehat{x}_{1000}$ to the limit $\widehat{x}_\infty$. 
The table further illustrates the interplay between risk aversion and the underlying parameter set. For the case of uniformly distributed parameters the level of risk aversion is of limited influence. In row 1 we find that the risk neutral estimator ($\theta = 0$) and the strongly risk averse estimator ($\theta=1000$) perform similarly well with respect to the expected value. In fact, the minimizer $\widehat{x}_0$ performs only about 
\begin{equation*}
    \frac{6.3765 - 6.0716}{6.3765} \approx 4.8 \%
\end{equation*}
better than $\widehat{x}_{1000}$. Similarly, row 5 shows that the performance of $\widehat{x}_0$ with respect to the essential supremum is almost as good as the one of the risk averse estimator $\widehat{x}_{1000}$. Specifically, the latter offers an improvement over the former of about
\begin{equation*}
    \frac{10.589 - 9.4986}{10.589} \approx 10.3 \%.
\end{equation*}
The results are drastically different for the second parameter set. In row 6 we observe that $\widehat{x}_0$ performs about 
\begin{equation*}
    \frac{9.2433 - 6.7095}{9.2433} \approx 27.4 \%
\end{equation*}
better than $\widehat{x}_{1000}$ when considering the expected value. On the other hand, in terms of the essential supremum the risk averse estimator offers an improvement of about 
\begin{equation*}
    \frac{38.15-11.945}{38.15} \approx 68.7 \%.
\end{equation*}
Summarizing the results we observe that the parameter set may play a crucial role towards the effect of the risk averse estimators. In our study we find that the relative differences for the two parameter sets differ by a factor of $6$. Further note that, by construction moving from risk neutral to strongly risk averse estimation improves the performance with respect to the essential supremum (see rows five and ten) while the performance with respect to the expectation might suffer, as shown in rows one and six. For this example, however, the improvement with respect to $\esssup$ is noticeably larger than the decrease in terms of $\mathbb{E}$.  

\begin{figure}
	\begin{subfigure}{0.49\textwidth}
        \includegraphics[scale = 1]{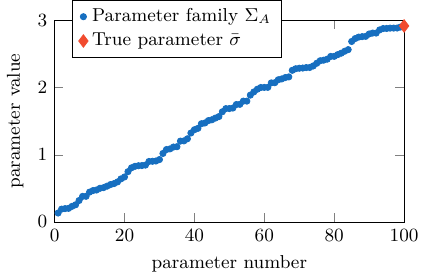}
		\caption{Uniform parameters}
		\label{subf: parameters_1}
	\end{subfigure}
	\begin{subfigure}{0.49\textwidth}
        \includegraphics[scale = 1]{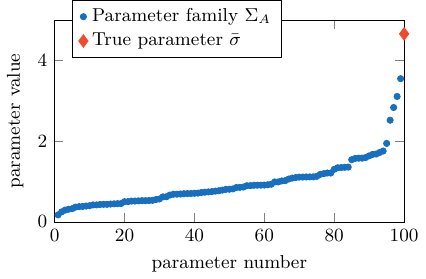}
		\caption{Log-normal parameters}
		\label{subf: parameters_2}
	\end{subfigure}
	\begin{subfigure}{0.49\textwidth}
        \includegraphics[scale = 1]{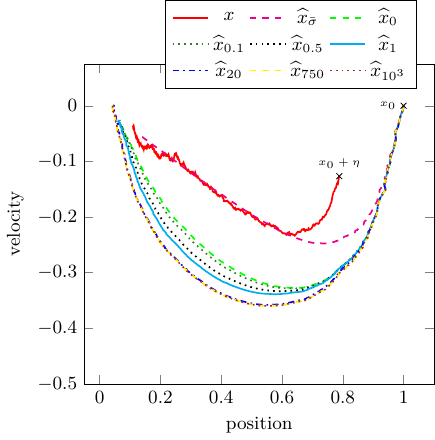}
        \caption{Trajectories for uniform parameters}
		\label{subf: phase_1}
	\end{subfigure}
	\begin{subfigure}{0.49\textwidth}
        \includegraphics[scale = 1]{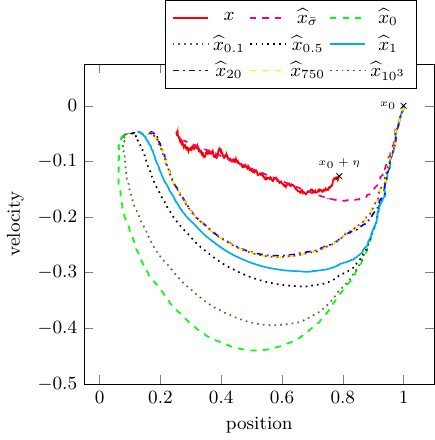}
		\caption{Trajectories for log-normal parameters}
		\label{subf: phase_2}
	\end{subfigure}
    \caption{Harmonic oscillator with uncertain damping parameter}
\end{figure}
\begin{table}
	\begin{center}
		\begin{NiceTabular}{l || l | | c c c c}
			\hline
			  & & $\theta = 0$ & $\theta = 0.5$ & $\theta = 20$ & $\theta = 1000$  \\
			\hline \hline 
            & Uniform parameters &&&& \\
            \hline \hline
			1 &$\int \mathbb{E}[\mathcal{V}_\sigma(t,\widehat{x}_\theta(t))] \, \mathrm{d}t$
            & $\textbf{6.0716}$ & $6.1175$ & $6.3561$ & $6.3765$   \\
            2 &$\int \rho_{0.5}(\mathcal{V}_\sigma(t,\widehat{x}_\theta(t)))\, \mathrm{d}t$ &
            $6.7182$ & $\textbf{6.6248}$ & $6.8269$ & $6.8508$ \\
            3 & $\int \rho_{20}(\mathcal{V}_\sigma(t,\widehat{x}_\theta(t)))\, \mathrm{d}t$ &
            $10.171$ & $9.6972$ & $\textbf{9.1291}$ & $9.148$ \\
			4& $\int \rho_{1000}(\mathcal{V}_\sigma(t,\widehat{x}_\theta(t)))\, \mathrm{d}t$ &
            $10.579$ & $10.111$ & $9.5187$ & $\textbf{9.4896}$  \\
            5& $\int \esssup \mathcal{V}_\sigma(t,\widehat{x}_\theta(t))\, \mathrm{d}t$ & 
            $10.589$ & $10.12$ & $9.5282$ & $\textbf{9.4986}$  \\
            \hline \hline
            & Log-normal parameters &&&& \\
            \hline \hline
			6& $\int \mathbb{E}[\mathcal{V}_\sigma(t,\widehat{x}_\theta(t))] \, \mathrm{d}t$ 
            & $\textbf{6.7095}$ & $7.7027$ & $9.2433$ & $9.2433$   \\
            7& $\int \rho_{0.5}(\mathcal{V}_\sigma(t,\widehat{x}_\theta(t)))\, \mathrm{d}t$ 
            & $20.581$ & $\textbf{8.521}$ & $9.7391$ & $9.738$ \\
            8& $\int \rho_{20}(\mathcal{V}_\sigma(t,\widehat{x}_\theta(t)))\, \mathrm{d}t$ &
            $37.676$ & $18.529$ & $\textbf{11.58}$ & $11.588$  \\
			9& $\int \rho_{1000}(\mathcal{V}_\sigma(t,\widehat{x}_\theta(t)))\, \mathrm{d}t$
            & $38.14$   & $19.007$  &  $11.958$ & $\textbf{11.935}$  \\
            10& $\int \esssup \mathcal{V}_\sigma(t,\widehat{x}_\theta(t))\, \mathrm{d}t$
            & $38.15$  &  $19.016$ &   $11.968$ & $\textbf{11.945}$  \\
            \hline
        \end{NiceTabular}
		\caption{Various risk measures integrated along selected state estimates}
		\label{tab: osc1}
	\end{center}
\end{table}
%

%
\subsection{Connected amplidynes with uncertain inductances}
%
For our second example we consider an electrical circuit amplifying a given input. 
The so-called amplidyne amplifies a given input and can be described via a linear dynamical system, see \cite[Ch.~1.12]{KwaSiv72}.{KwaSiv72}.
For our experiments we consider two such amplidynes that are connected such that the output of the first one acts as an input of the second one. The total system then consists of four components and is modeled as 
\begin{equation*}
	\begin{alignedat}{2}
		\dot{x}(t) &= 
		\begin{bmatrix}
			- \frac{\rho_1}{L_1} & 0 & 0 & 0\\
			\frac{k_1}{L_2} & - \frac{\rho_2}{L_2} & 0 & 0 \\
			0 & \frac{k_2}{L_3} & - \frac{\rho_3}{L_3} & 0 \\ 
			0 & 0 & \frac{k_3}{L_4} & - \frac{\rho_4}{L_4}
		\end{bmatrix}
		x(t)
		+
		\begin{bmatrix}
			\frac{e_0 (t)}{L_1} \\ 0 \\ 0 \\ 0
		\end{bmatrix}
		+
		\begin{bmatrix}
			\frac{1}{L_1} \\ 0 \\ 0 \\ 0
		\end{bmatrix}
		v(t)
		~~~~~&& t \in (0,T),\\
		x(0) &=
		x_0 + \eta, && \\
		y(t) 
		&= 
		\begin{bmatrix}
			0 & 0 & 0 & k_4
		\end{bmatrix}
		x(t)
		+ \mu(t)
		~~~~~&& t \in (0,T),
	\end{alignedat}
\end{equation*}
where $x(t) \in \mathbb{R}^4$ encodes the current of the individual components. The resistances and inductances of the individual components are given by $\rho_i > 0$ and $L_i > 0$, respectively. Further for $i=1,2,3,4$ we have  $e_i = k_i x_i$, i.e., $k_i > 0$ describe the proportion of the currents $x_i$ and the voltages $e_i$. By $e_0$ we denote the known, time-dependent input entering the first component. It is subject to the disturbance $v$. Finally, the measured output is the output of the second amplidyne $e_4 = k_4 x_4$.
The known forcing term given by $e_0$ can be incorporated in our formulation of the Kalman filter in a straightforward fashion.
Our numerical experiments are conducted with $\rho_1 = \rho_3 = 5$, $\rho_2 = \rho_4 = 10$, $k_1 = k_3 = 20$, $k_2 = k_4 = 50$, $L_1 = L_3 = 0.5$, $T=10$, and $e_0(t) \equiv 1$ and undisturbed initial state $x_0 = \begin{bmatrix} 0.5 & 1 & 10 & 20 \end{bmatrix}^\top$. The uncertainty of the system lies in the inductances $\sigma = \begin{bmatrix} L_2 & L_4 \end{bmatrix}^\top \in \mathbb{R}^2$. 
As noise covariances we use 
\begin{equation*}
	\begin{aligned}
		\Gamma &= 0.25~ \mathrm{diag} (\vert x_{0,1} \vert,\vert x_{0,2} \vert,\vert x_{0,3} \vert,\vert x_{0,4} \vert) = \mathrm{diag} (0.125,0.25,2.5,5),\\
		R &= (0.1 \vert e_0(0) \vert)^2 = 0.01,\\
		Q &= (0.1 * 400 \vert e_0(0) \vert)^2 = 1600,
	\end{aligned}
\end{equation*}
where the different magnitudes are due to the scales of the different components. 

Again we construct two parameter sets to compare their effects on the risk averse state estimators. Both contain $N_A = 100$ elements sampled according to given bivariate distributions. The first set of parameters is sampled from a uniform distribution on the rectangle $[10,40]^2$ and illustrated in \Cref{subf: param_3}. The second parameter set is sampled from a Gaussian mixture characterized via 
\begin{equation*}
    0.95 \, \mathcal{N}\left(
    \begin{bmatrix}
        15 \\ 35
    \end{bmatrix}
    ,
    \begin{bmatrix}
        2 & 0 \\ 0 & 2
    \end{bmatrix}\right)
    + 0.05 \, \mathcal{N}\left(
    \begin{bmatrix}
        35 \\ 15
    \end{bmatrix}
    ,
    \begin{bmatrix}
        1 & 0 \\ 0 & 1
    \end{bmatrix}\right),
\end{equation*}
i.e., every drawn sample has a chance of $0.95$ to be drawn from a normal distribution with mean $\begin{bmatrix} 15 & 35 \end{bmatrix}^\top$ and covariance $2 \mathrm{Id}$ and a $0.05$ chance to be drawn from a normal distribution with mean $\begin{bmatrix} 35 & 15 \end{bmatrix}^\top$ and covariance $ \mathrm{Id}$. The resulting parameter set is shown in \Cref{subf: param_4}. For both sets the parameter $\sigma = (L_2,L_4)$ that maximizes $L_2 - L_4$ is set to be the true underlying parameter. 

To illustrate the qualitative behavior of the trajectories we plot the second component of the resulting state $x$ and the state estimators in \Cref{subf: comp_2} and \Cref{subf: comp_2d}. Again we observe satisfying results obtained from the Kalman filter constructed based on the hidden parameter. Further, for both parameter sets the risk averse estimators seem to outperform the risk neutral option in terms of euclidean distance.
We note that for this example convergence for $\theta \to \infty$ is approached for far lower $\theta$ when compared to the oscillator, showing that the effective range of the risk aversion parameter is highly problem dependent.  
To further investigate the performances of the risk neutral estimator $\widehat{x}_0$ and the risk averse estimator $\widehat{x}_4$ we present the energies according to each parameter $\sigma \in \Sigma_A$ or equivalently for each $\sigma_k$, $k = 1,\dots,N_A$. More precisely, in \Cref{subf: ens_en1} and \Cref{subf: ens_en2} we plot the values
\begin{equation*}
    \mathcal{V}_{\sigma_k} (t,\widehat{x}_\theta(t))
    =
    \mathcal{V}_k (t,\widehat{x}_\theta(t))
\end{equation*}
along $t \in [0,T]$ and for $\theta = 0$ and $\theta = 4$. Additionally, we present the resulting risk measures
\begin{equation*}
    \mathbb{E}\left[ 
    \mathcal{V}_{\sigma_k} (t,\widehat{x}_\theta(t)) 
    \right]
    ~~~
    \text{and}
    ~~~
    \max_k \mathcal{V}_{\sigma_k} (t,\widehat{x}_\theta(t)),
\end{equation*}
again for $\theta = 0$ and $\theta = 4$. Comparing the dashed lines in \Cref{subf: ens_en1} we find that for the uniform parameter set $\widehat{x}_4$ significantly outperforms $\widehat{x}_0$ with respect to the maximum over the family members, showing that for this example the risk averse state estimation works exactly as intended. Interestingly, $\widehat{x}_4$ still performs relatively well with respect to the expectation when compared to the minimizer $\widehat{x}_0$, cf.~solid orange and dotted red line. Turning to \Cref{subf: ens_en2} we observe that the risk aversion takes an even stronger effect. It appears that the energies corresponding to the four outliers in the parameter set, cf.~\Cref{subf: param_4}, evaluated along $\widehat{x}_0$ clearly stand out resulting in a poor performance with respect to the maximum. On the other hand, in this example the relative difference of the expected energies along $\widehat{x}_4$ and $\widehat{x}_0$ is larger when compared to the uniform parameter set. 

\begin{figure}
	\begin{subfigure}{0.49\textwidth}
        \includegraphics[scale = 1]{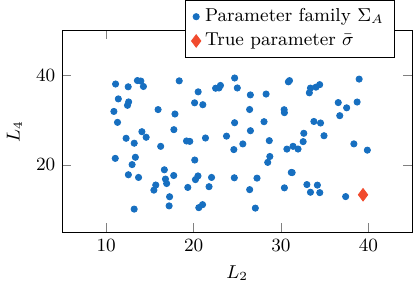}
		\caption{Uniform parameters}
		\label{subf: param_3}
	\end{subfigure}
	\begin{subfigure}{0.49\textwidth}
        \includegraphics[scale = 1]{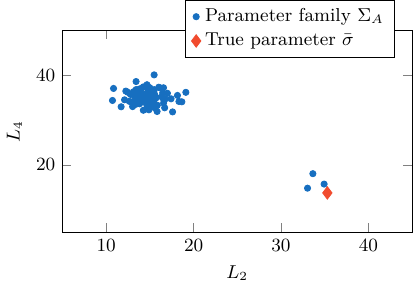}
		\caption{Parameters via Gaussian mixture}
		\label{subf: param_4}
	\end{subfigure}
	\begin{subfigure}{0.49\textwidth}
        \includegraphics[scale = 1]{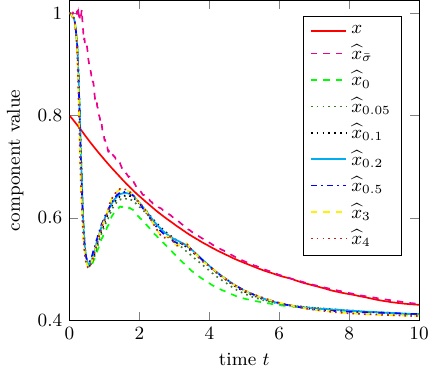}
		\caption{Second component along time}
		\label{subf: comp_2}
	\end{subfigure}
	\begin{subfigure}{0.49\textwidth}
        \includegraphics[scale = 1]{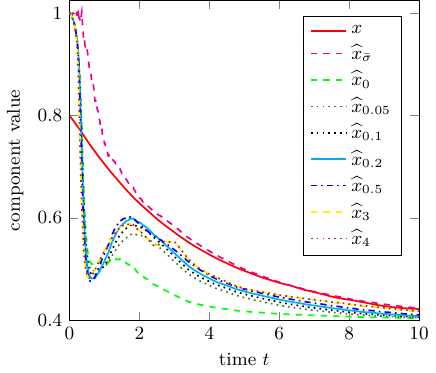}
		\caption{Second component along time}
		\label{subf: comp_2d}
	\end{subfigure}
	\begin{subfigure}{0.49\textwidth}
        \includegraphics[scale = 1]{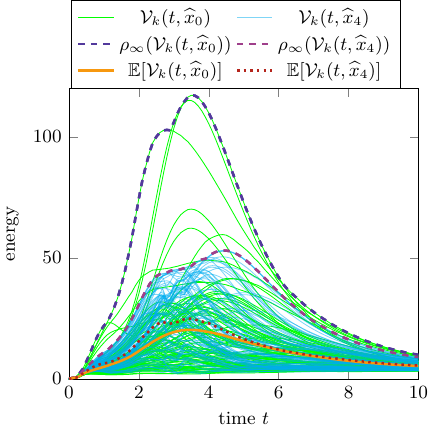}
		\caption{Energies via uniform parameters}
		\label{subf: ens_en1}
	\end{subfigure}
	\begin{subfigure}{0.49\textwidth}
        \includegraphics[scale = 1]{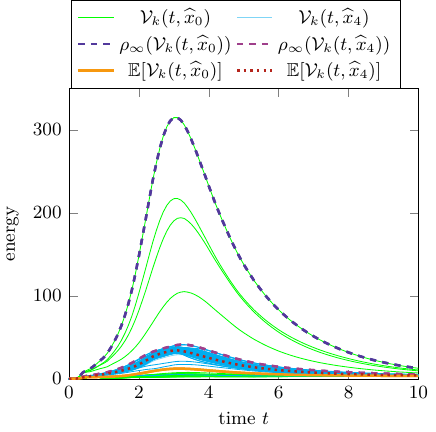}
		\caption{Energies via Gaussian mixture}
		\label{subf: ens_en2}
	\end{subfigure}
    \caption{Amplidyne with uncertain inductance}
\end{figure}
%

%
\section{Conclusion}
%
This work presents a novel approach for risk averse state estimation under uncertainty for linear, time-invariant, finite-dimensional systems. Investigating analogous concepts for time-dependent, nonlinear systems as well as for systems governed by PDEs could be the focus of future research.

%
\section*{Acknowledgement}
We thank P.~Guth (RICAM Linz) for many fruitful discussions on the entropic risk measure.  
%

%
%

\bibliographystyle{siam}
\bibliography{references} 

\appendix

%
\section{Technical proofs}
%

%
\subsection{Technical auxiliary results}\label[subsection]{subs: techn_results}
%
This section contains auxiliary results required for the analysis of this work. The first lemma is concerned with the eigenvalues of the precision matrices and their impact on the weighted norms. 
\begin{lemma}\label[lemma]{lem: estWeightNorm}
    Let $k \in \{ 1,\dots,N \}$, $t \in [0,T]$, and $\lambda_k(t) \in \mathbb{R}$ be an eigenvalue of $P_k(t)$, where $P_k$ are the precision matrices as introduced in \Cref{subsec: ModUnc}.  
    There exist real numbers $\lambda_\mathrm{min},\lambda_\mathrm{max} > 0$ independent of $t$ and $k$ such that
    \begin{equation*}
        \lambda_\mathrm{min} \leq \lambda_k(t) \leq \lambda_\mathrm{max}.
    \end{equation*}
    In particular for every $x \in \mathbb{R}^n$ it holds
    \begin{equation*}
        \lambda_\mathrm{min} \, \Vert x \Vert^2
        \leq 
        \Vert x \Vert_{P_k(t)}^2
        \leq 
        \lambda_\mathrm{max} \, \Vert x \Vert^2.
    \end{equation*}
\end{lemma}
\begin{proof}
    The boundedness of the eigenvalues follows from their continuous dependence on the matrix entries and the regularity of $P_k$ presented in \eqref{eq: P_smooth}. The positivity of the lower bound is a consequence of the positive definiteness of $P_k$. The estimates for the weighted norm are a consequence of the symmetry of $P_k$.
\end{proof}
Next we estimate the spectral norm of the inverse of a  convex combination of the precision matrices.
\begin{lemma}\label[lemma]{lem: InvOfPrecSumBD}
    Let $P_k$ be the precision matrices and let $S \subset \{1,\dots,N\}$ and $(\gamma_k )_{k \in S}$ be such that $\sum_{k \in S} \gamma_k =1$. Then for every $t \in [0,T]$ the matrix $\sum_{k \in S} \gamma_k P_k(t)$ is non singular and
    \begin{equation*}
        \Vert ( \sum_{k \in S} \gamma_k P_k(t) )^{-1} \Vert_2
        \leq \lambda_\mathrm{min}^{-1}.
    \end{equation*}
\end{lemma}
\begin{proof}
    For every $x \in\mathbb{R}^n$ we have
    \begin{equation*}
        \left\langle x, \sum_{k \in S} \gamma_k P_k(t) x \right\rangle
        \geq
        \sum_{k \in S} \gamma_k \lambda_\mathrm{min} \Vert x \Vert^2
        =
        \lambda_\mathrm{min} \Vert x \Vert^2.
    \end{equation*}
    The assertion follows with the Lax-Milgram Lemma \cite[Thm.~1.1.3 {\&} Rem.~1.1.3]{Cia78}.
\end{proof}
%

%
\subsection{Proof of \Cref{thm: regularity_x_theta}}\label[subsection]{subs: proof_reg_EntrR_Min}
%

This subsection contains the proof of the regularity of the estimator minimizing the entropic risk of the energy. 
We begin by showing the result for continuous outputs $y$. The result for $p$-integrable outputs will follow using an approximation by continuous elements.

Throughout this subsection we use the following notation. For a given $y \in \mathcal{L}_{0,T}^2$ and $\theta \in (0,\infty)$ we denote by $\widehat{x}_k[y]$ the unique solution associated with \eqref{eq: KF_k}, we set
\begin{equation}\label{eq: r[y]}
    r_k[y](t) = \int_0^t \Vert y(s) - C \widehat{x}_k[y](s) \Vert_{Q_k^{-1}}^2 \, \mathrm{d}s.
\end{equation}
By $\widehat{x}_\theta[y]$ we denote the associated minimizer of the entropic risk as defined in \Cref{def:minEntrR}.
We begin with a a uniform bound of the estimator.
\begin{lemma}\label[lemma]{lem: UnifBoundTheta}
    Let $t \in [0,T]$, $\theta \in (0,\infty)$, and $y \in \mathcal{L}_{0,T}^2$. There exists a constant $\widehat{C} > 0$ independent of $t$, $\theta$, and $y$ such that
    \begin{equation*}
        \Vert \widehat{x}_\theta[y] (t) \Vert^2 
        \leq
        \widehat{C} \left(
        \Vert x_0 \Vert^2 + \Vert y \Vert_{\mathcal{L}_{0,T}^2}^2 \right).
    \end{equation*}
    In particular $\widehat{x}_\theta[y] \in \mathcal{L}_{0,T}^\infty$.
\end{lemma}
\begin{proof}
    Using the characterization of $\widehat{x}_\theta[y](t)$ given in \eqref{eq: charMinEntrR2} the proof can be conducted via the arguments used in the proof of \Cref{cor: min_inf_essBd}. The particular bound in terms of $x_0$ and $y$ follows from the estimate of $\widehat{x}_k$ obtained via a Gronwall estimate.  
\end{proof}
Next we show a regularity result for $\widehat{x}_\theta[y]$ under the assumption of continuous data. To that end we define the mapping
\begin{equation}\label{eq: implRiskMin}
\begin{aligned}
    &G[\theta,y] \colon [0,T] \times \mathbb{R}^n \to \mathbb{R}^n\\
    &G[\theta,y](t,x) \coloneqq \sum_{k=1}^N e^{ \theta \Vert x - \widehat{x}_k[y](t) \Vert_{P_k(t)}^2 + \theta r_k[y](t) } P_k(t) (x - \widehat{x}_k[y](t)). 
\end{aligned}
\end{equation}
The following lemma investigates its regularity depending on the smoothness of the output. Again for $x \in \mathbb{R}^n$ we denote $x_\otimes^2 = x \otimes x \in \mathbb{R}^{n,n}$.
\begin{lemma}\label[lemma]{lem: regImpl}
    Let $\theta \in (0,\infty)$ and $y \in \mathcal{L}_{0,T}^2$. 
    \vspace{1mm}\\
    \begin{enumerate}
        \item For every $t \in [0,T]$ the mapping $G[\theta,y](t,\cdot)$ is of class $C^\infty$. 
        Its first derivative is given by
        \begin{equation}\label{eq: F_x_derivative}
        \begin{aligned}
            &D_x G[\theta,y](t,x)\\
            &=
            \sum_{k=1}^N  e^{\theta \Vert x - \widehat{x}_k[y](t) \Vert_{P_k(t)}^2 + \theta r_k[y](t)}
            \left(
            P_k(t) + 
            2 \theta \left[P_k(t) (x - \widehat{x}_k[y](t)) \right]_\otimes^2 
            \right).
        \end{aligned}
        \end{equation}
        \vspace{1mm}
        \item For every $t \in [0,T]$ and $x \in \mathbb{R}^n$ the matrix $D_x G[\theta,y](x,t)$ is invertible. Further there exists a constant $c$ independent of $\theta$, $y$, $t$, and $x$ such that
        \begin{equation*}
            \Vert D_x G[\theta,y](t,x)^{-1} \Vert_2 \leq c.
        \end{equation*}
        \vspace{1mm}
        \item Assuming continuity of the output, i.e., $y \in C([0,T];\mathbb{R}^r)$ it holds that $G[\theta,y]$ is continuously differentiable in $[0,T] \times \mathbb{R}^n$.
        Its partial time derivative is given as
        \begin{equation}\label{eq: F_t_derivative}
        \begin{aligned}
            \partial_t G[\theta,y](t,x) 
            =  
            \sum_{k=1}^N& e^{\theta \Vert x - \widehat{x}_k[y](t) \Vert_{P_k(t)}^2 + \theta r_k[y](t)}
            \left( 
            - P_k(t) \dot{\widehat{x}}_k[y](t)
            + \dot{P}_k(t) (x - \widehat{x}_k[y](t) )
            \right.\\
            &- 2 \theta \langle x - \widehat{x}_k[y](t) , P_k(t) \dot{\widehat{x}}_k[y](t) \rangle  \, P_k(t) (x - \widehat{x}_k[y](t))\\
            &\left.
            \vphantom{- P_k(t) \dot{\widehat{x}}_k[y](t)
            + \dot{P}_k(t) (x - \widehat{x}_k[y](t) )}
            + \theta \, \Vert y(t) - C \widehat{x}_k[y](t) \Vert_{Q_k^{-1}}^2
            \, P_k(t) (x - \widehat{x}_k[y](t))
            \right).
        \end{aligned}
        \end{equation}   
    \end{enumerate}
\end{lemma}
\begin{proof}
    The regularity with respect to $x$ and the formula for the associated partial derivative are a direct consequence of the chain rule and the regularity of the exponential and the squared weighted norm proving (i). 

    To show (ii) let $t \in [0,T]$ and $x,z \in \mathbb{R}^n$ be arbitrary. Utilizing \eqref{eq: F_x_derivative}, the positive definiteness of $P_k(t)$, and the fact that $e^a \geq 1$ for $a \geq 0$ we find 
    \begin{equation*}
    \begin{aligned}
        &z^\top D_x G[\theta,y](t,x) z
        \geq
        \sum_{k=1}^N z^\top P_k(t) z
        +
        2 \theta 
        \left\vert z^\top P_k(t) (x - \widehat{x}_k[y](t)) \right\vert^2
        \geq 
        \min_{t \in [0,T]} z^\top P_1(t) z \eqqcolon c^\prime,
    \end{aligned}
    \end{equation*}
    where the continuity of $P_1$ ensures that the minimum on the right hand side is attained. The Lax-Milgram Lemma \cite[Thm.~1.1.3 {\&} Rem.~1.1.3]{Cia78} yields the invertibility and the bound with $c \coloneqq \tfrac{1}{c^\prime}$. 

    Turning to the regularity in time we note that $\Pi_k$ given as a solution of \eqref{eq: Ricc_k} is continuously differentiable. Further, for $y \in C([0,T];\mathbb{R}^r)$ we find that the right hand side of \eqref{eq: KF_k} is continuous. For the associated weak solution it follows $\widehat{x}_k[y] \in C^1([0,T];\mathbb{R}^n)$.  
    Additionally, the integrant in \eqref{eq: r[y]} is continuous, ensuring $r_k[y] \in C^1([0,T];\mathbb{R})$.
    It follows that $t \to G[\theta,y](t,x)$ is a composition of continuously differentiable functions. Hence for $y \in C([0,T];\mathbb{R}^r)$ the formula for the time derivative follows with the chain rule. Since both partial derivatives exist and are continuous in $[0,T] \times \mathbb{R}^n$, we obtain $G[\theta,y] \in C^1([0,T] \times \mathbb{R}^n;\mathbb{R}^n)$.
\end{proof}
With these results at hand we can show the first regularity result\\
\vspace{1mm}\\
\textbf{Proof of \Cref{thm: regularity_x_theta}(ii)}

\begin{proof}
    The assertion is a direct consequence of the implicit function theorem. \Cref{lem: charEntrMin} shows that $\widehat{x}_\theta [y]$ is implicitly defined by \eqref{eq: implRiskMin}, i.e., for a given $t \in [0,T]$ we have that $x = \widehat{x}_\theta [y](t)$ is the unique solution of 
    \begin{equation*}
        0 = G[\theta,y](t,x).
    \end{equation*}
    With the results from \Cref{lem: regImpl} we can apply the implicit function theorem \cite[Thm.~4.E]{Zei95AMS109}
    and obtain $\widehat{x}_\theta[y] \in C^1([0,T];\R^n)$ and the following formula for the derivative
    \begin{equation}\label{eq: Dt_x_theta}
        \dot{\widehat{x}}_\theta[y](t)
        =
        - D_x G[\theta,y](t,\widehat{x}_\theta[y](t))^{-1}
        \partial_t G[\theta,y](t,\widehat{x}_\theta[y](t)).
    \end{equation}
    In light of the identities \eqref{eq: F_x_derivative} and \eqref{eq: F_t_derivative} this is the formula for the derivative as announced in \Cref{thm: regularity_x_theta}.
\end{proof}
This result will be carried over to less regular measurements $y$ via a density argument. To that end we first ensure that the terms appearing in the energy $\mathcal{V}_k$ depend continuously on the output in an appropriate sense.
\begin{lemma}\label[lemma]{lem: ContInYEnergy}
    For all $k \in \{ 1,\dots,N \}$, $t \in (0,T]$, and $s \in [0,t]$ the mappings
    \begin{equation*}
        y \in \mathcal{L}_{0,t}^2 \mapsto \widehat{x}_k[y] \in \mathcal{H}_0^t,
        ~~~~~  ~~~~~~
        y \in \mathcal{L}_{0,t}^2 \mapsto \widehat{x}_k[y](s) \in \mathbb{R}^n
    \end{equation*}
    and
    \begin{equation*}
        y \in \mathcal{L}_{0,t}^2 \mapsto r_k[y] \in C([0,t];\mathbb{R}),
        ~~~~~  ~~~~~~
        y \in \mathcal{L}_{0,t}^2 \mapsto r_k[y](s) \in \mathbb{R},
    \end{equation*}
    are continuous. 
\end{lemma}
\begin{proof}
    Let $k$ and $t \in [0,T]$  be fixed and let $(y_j) \subset \mathcal{L}_{0,T}^2$ be a sequence converging to some $\bar{y} $ in $\mathcal{L}_{0,T}^2$. Using a standard Gronwall argument one can show existence of a constant $c >0$ such that
    \begin{equation*}
        \Vert \widehat{x}_k[y_j] - \widehat{x}_k[\bar{y}] \Vert_{\mathcal{H}_0^T}
        \leq 
        c \Vert y_j - \bar{y} \Vert_{\mathcal{L}_{0,t}^2},
    \end{equation*}
    implying continuity of the first mapping. The claim for the second mapping follows from the continuous embedding $ \mathcal{H}_0^t \hookrightarrow C([0,t];\mathbb{R}^n) $ and the continuity of the point evaluation. 

    Now let $y_1,y_2 \in \mathcal{L}_{0,t}^2$. We find 
    \begin{equation*}
    \begin{aligned}
        &\sup_{s \in [0,t]} \vert r_k[y_j](s) - r_k[\bar{y}](s) \vert\\
        &\leq
        \sup_{s \in [0,t]} \int_0^s \vert
        \Vert y_j - C\widehat{x}_k[y_j \Vert_{Q_k^{-1}}^2 
        - \Vert \bar{y} - C\widehat{x}_k[\bar{y}] \Vert_{Q_k^{-1}}^2 
        \vert \, \mathrm{d}\tau\\
        &\leq
        \sup_{s \in [0,t]} \int_0^s
        \vert \langle
        y_j - \bar{y} + C (\widehat{x}_k[y_j] - \widehat{x}_k[\bar{y}]),
        y_j + \bar{y} + C (\widehat{x}_k[y_j] + \widehat{x}_k[\bar{y}])
        \rangle \vert
        \, \mathrm{d}\tau\\
        &\leq
        c
        \left(
        \Vert y_j - \bar{y} \Vert_{\mathcal{L}_{0,t}^2} 
        + \Vert \widehat{x}_k[y_j] - \widehat{x}_k[\bar{y}] \Vert_{\mathcal{L}_{0,t}^2}
        \right)
        \left(
        \Vert y_j + \bar{y} \Vert_{\mathcal{L}_{0,t}^2} 
        + \Vert \widehat{x}_k[y_j] + \widehat{x}_k[\bar{y}] \Vert_{\mathcal{L}_{0,t}^2}
        \right)
        ,
    \end{aligned}
    \end{equation*}
    where the $\tau$ dependence of the integrand is suppressed in the notation and $c$ is an appropriate constant. Due to the convergence of $y_j$ to $\bar{y}$ the right hand side converges to zero ensuring continuity of the third mapping. The assertion for the fourth mapping follows directly. 
\end{proof}
With this result at hand we are able to show that the estimator depends continuously on the output in an $L^2$ sense. 
\begin{proposition}\label[proposition]{prop: LpConv_theta}
    Let $\theta \in (0,\infty)$ and $(y_j) \subset \mathcal{L}_{0,T}^2$ be a sequence converging to some $\bar{y} $ in $\mathcal{L}_{0,T}^2$. Then it holds that
    \begin{equation*}
    \begin{alignedat}{2}
        \forall t \in [0,T]~~~~~
        \Vert \widehat{x}_\theta[y_j](t) - \widehat{x}_\theta[\bar{y}](t) \Vert
        &\to 0
        ~&&\text{as}~ j \to \infty,\\
        \Vert \widehat{x}_\theta[y_j] - \widehat{x}_\theta[\bar{y}] \Vert_{\mathcal{L}_{0,T}^1}
        &\to 0
        ~&&\text{as}~ j \to \infty,
    \end{alignedat}
    \end{equation*}
\end{proposition}
\begin{proof}
    Assume for now that $\widehat{x}_\theta[y_j]$ converges to $\widehat{x}_\theta[\bar{y}]$  pointwise  for all $t \in [0,T]$. 
    We establish an integrable upper bound for $\Vert \widehat{x}_\theta[y_j](t) \Vert$ to apply the dominated convergence. The convergence of the sequence $(y_j)$ implies the existence of  a constant $C$ such that   $\Vert y_j \Vert_{\mathcal{L}_{0,T}^2} \leq C$ for all $j\in \mathbb{N}$. Utilizing \Cref{lem: UnifBoundTheta} we obtain that for all $t$ it holds
    \begin{equation}\label{eq:aux2}
        \Vert \widehat{x}_\theta[y_j](t) \Vert
        \leq 
        \sqrt{
        \widehat{C} \left( \Vert x_0 \Vert^2 + \Vert y_j \Vert_{\mathcal{L}_{0,T}^2}^2 \right)}
        \leq 
        \sqrt{
        \widehat{C} \left( \Vert x_0 \Vert^2 + C^2 \right)}.
    \end{equation}
     Since the right hand side is constant, the integrable upper bound is found and the dominated convergence theorem  yields the desired convergence in $\mathcal{L}_{0,T}^p$.

    It remains to show the pointwise convergence. Let $t \in [0,T]$ be fixed. 
    By \eqref{eq:aux2} the family $\{ \widehat{x}_\theta[y_j](t)\}_{j \in \mathbb{N}}$ is bounded in $\mathbb{R}^n$. Therefore each arbitrary subsequence admits another subsequence (also denoted by $\widehat{x}_\theta[y_j](t)$) that converges to some $\bar{x} \in \mathbb{R}^n$.  Next we recall from \Cref{lem: charEntrMin} that for all $j$
    \begin{equation*}
    \begin{aligned}
        \widehat{x}_\theta[y_j] (t)
        &=
        \left( \sum_{k=1}^N e^{\theta  \Vert \widehat{x}_\theta[y_j] (t) - \widehat{x}_k[y_j](t) \Vert_{P_k(t)}^2 + \theta r_k[y_j](t) } P_k(t) \right)^{-1}\\
        &\sum_{k=1}^N e^{\theta  \Vert \widehat{x}_\theta[y_j] (t) - \widehat{x}_k[y_j](t) \Vert_{P_k(t)}^2 + \theta r_k[y_j](t) } P_k(t) \widehat{x}_k[y_j](t)
    \end{aligned}
    \end{equation*}
   holds. Considering the convergent   subsequence and utilizing the regularity results established in \Cref{lem: ContInYEnergy} and the continuity of the squared norm, the exponential, and the matrix inverse we find that the right hand side converges and we obtain
    \begin{equation*}
        \bar{x}
        =
        \left( \sum_{k=1}^N e^{\theta  \Vert \bar{x} - \widehat{x}_k[\bar{y}](t) \Vert_{P_k(t)}^2 + \theta r_k[\bar{y}](t) } P_k(t) \right)^{-1}
        \sum_{k=1}^N e^{\theta  \Vert \bar{x} - \widehat{x}_k[\bar{y}](t) \Vert_{P_k(t)}^2 + \theta r_k[\bar{y}](t) } P_k(t) \widehat{x}_k[\bar{y}](t).
    \end{equation*}    
    By \Cref{lem: charEntrMin} it follows that $\bar{x} = \widehat{x}_\theta[\bar{y}](t)$ and we have shown that any subsequence of $\widehat{x}_\theta[y_j](t)$ admits a subsequence converging to $\widehat{x}_\theta[\bar{y}](t)$. By the subsequence principle the pointwise convergence follows and the proof is concluded.
\end{proof}
Next we show convergence of the time derivatives of an approximating sequence.
\begin{lemma}\label[lemma]{lem: conv_derivative}
    Let $\theta \in (0,\infty)$ and $(y_j) \subset C([0,T];\mathbb{R}^r)$ be a sequence converging to some $\bar{y} $ in $\mathcal{L}_{0,T}^2$.
    Define $\bar{x}$ as 
    \begin{equation}\label{eq: proof_x_bar}
        \bar{x}(t) = D_x G[\theta,\bar{y}](t,\widehat{x}_\theta[\bar{y}](t))^{-1}
        \partial_t G[\theta,\bar{y}] (t,\widehat{x}_\theta[\bar{y}](t)).
    \end{equation}
    Then $\bar{x} \in \mathcal{L}_{0,T}^1 $ and $(y_j)$ admits a subsequence (denoted by $(y_j)$) such that
    \begin{equation}\label{eq: L1_conv}
        \int_0^T \Vert
        \dot{\widehat{x}}_\theta[y_j](t) 
        - \bar{x}(t)
        \Vert \, \mathrm{d}t
        \to
        0
        ~~~~~
        \text{for}~
        j \to \infty.
    \end{equation}
\end{lemma}
\begin{proof}
    The assertion that $\bar{x} \in \mathcal{L}_{0,T}^1$ follows from \Cref{lem: regImpl} (ii) and (iii).
    Turning to \eqref{eq: L1_conv} we first note that
    by assumption we have convergence of $y_j$ in $\mathcal{L}_{0,T}^2$ and by \Cref{lem: ContInYEnergy} we have convergence of $\dot{\widehat{x}}_k[y_j]$ in $\mathcal{L}_{0,T}^2$. Hence \cite[Cor.~2.32]{Fol99} allows the extraction of a subsequence such that $y_j$ and $\dot{\widehat{x}}_k[y_j]$ converge pointwise almost everywhere. Further by \Cref{lem: ContInYEnergy} and \Cref{lem: conv_derivative} we have pointwise convergence of $\widehat{x}_k[y_j]$, $r_k[y_j]$, and $\widehat{x}_\theta[y_j]$. With \eqref{eq: F_x_derivative}, \eqref{eq: F_t_derivative}, and \eqref{eq: Dt_x_theta} it follows that for almost all $t \in [0,T]$ it holds
    \begin{equation*}
        \dot{\widehat{x}}_\theta[y_j](t)
        =
        - D_x G[\theta,y_j](t,\widehat{x}_\theta[y_j](t))^{-1}
        \partial_t G[\theta,y_j](t,\widehat{x}_\theta[y_j](t))
        \to \bar{x}(t)
        ~~~~~
        \text{for}~j \to \infty.
    \end{equation*}
    An integrable function dominating $\dot{\widehat{x}}_\theta[y_j]$ can be derived using the boundedness of $\Vert y_j \Vert_{\mathcal{L}_{0,T}^2}$, \Cref{lem: UnifBoundTheta}, \Cref{lem: ContInYEnergy}, and the continuity of $P_k$ and $\dot{P}_k$. The assertion follows with dominated convergence.
\end{proof}
We now identify the limit $\bar{x}$ as the weak derivative of $\widehat{x}_\theta[\bar{y}]$ to show the asserted regularity of the estimator.
\vspace{1mm}\\
\textbf{Proof of \Cref{thm: regularity_x_theta}(i)}
\begin{proof}
    We begin by showing the result for $p=1$. Hence let $\bar{y} \in \mathcal{L}_{0,T}^2$. From \Cref{lem: UnifBoundTheta} we know that $\widehat{x}_\theta[\bar{y}] \in \mathcal{L}_{0,T}^\infty$. To obtain the assertion it remains to show that its weak derivative exists and is integrable. 
    We denote $\bar{x} \in \mathcal{L}_{0,T}^1$ as in \eqref{eq: proof_x_bar}.
    Then the density of $C([0,T];\mathbb{R}^r)$ in $\mathcal{L}_{0,T}^2$ together with \Cref{prop: LpConv_theta} and \Cref{lem: conv_derivative} ensures existence of a sequence $(y_j) \subset C([0,T];\mathbb{R}^r)$ such that
    \begin{alignat}{3}
        \widehat{x}_\theta[y_j] 
        &\to 
        \widehat{x}_\theta[\bar{y}]
        ~~~~~&&\text{for}~ j \to \infty
        ~~~~~~~~~~
        &\text{in}~
        \mathcal{L}_{0,T}^1, \label{eq: proofConv1}\\
        \dot{\widehat{x}}_\theta[y_j] 
        &\to 
        \bar{x}
        ~&&\text{for}~ j \to \infty
        ~~~~~~~~~~
        &\text{in}~
        \mathcal{L}_{0,T}^1. \label{eq: proofConv2}
    \end{alignat}
    Now for any test function $\varphi \in C_0^\infty([0,T];\mathbb{R}^n)$ we obtain
    \begin{equation*}
    \begin{aligned}
        \int_0^T \langle \widehat{x}_\theta [\bar{y}] (t), \dot{\varphi}(t) \rangle \, \mathrm{d}t 
        &=
        \lim_{j \to \infty} 
        \int_0^T \langle \widehat{x}_\theta [y_j] (t), \dot{\varphi}(t) \rangle \, \mathrm{d}t \\
        &=
        - \lim_{j \to \infty} 
        \int_0^T \langle \dot{\widehat{x}}_\theta [y_j] (t), \varphi(t) \rangle \, \mathrm{d}t
        =
        - \int_0^T \langle \bar{x} (t), \varphi(t) \rangle \, \mathrm{d}t,
    \end{aligned}
    \end{equation*}
    where the first equality is justified by \eqref{eq: proofConv1}, the second equality follows from \Cref{thm: regularity_x_theta}(iii) and partial integration, and the third equality holds due to \eqref{eq: proofConv2}. 
    By definition of the weak derivative we have shown that $\bar{x} \in \mathcal{L}_{0,T}^1$ is the weak derivative of $\widehat{x}_\theta[\bar{y}]$, in short $\dot{\widehat{x}}_\theta [\bar{y}] = \bar{x}$. It follows that $\widehat{x}_\theta [\bar{y}] \in W^{1,1}(0,T;\mathbb{R}^n)$. Further, by the definition of $\bar{x}$ the formula for the weak derivative given in \eqref{eq: min_entrR_TimeDer} is proven to hold for square integrable outputs.
    
    It remains to show the result for $2 \leq p < \infty$. For $y \in \mathcal{L}_{0,T}^{2p} \hookrightarrow \mathcal{L}_{0,T}^{2}$ the first part of this proof ensures weak differentiability in $\mathcal{L}_{0,T}^1$, and \eqref{eq: min_entrR_TimeDer} for its derivative. From this formula it can subsequently be confirmed that $\dot{\widehat{x}}_\theta[y] \in \mathcal{L}_{0,T}^p$ and the assertion is shown. 
\end{proof}

\end{document}